\documentclass[11pt]{article}
\usepackage[all]{xy}
\usepackage{amsmath}
\usepackage{amsfonts}
\usepackage{amssymb}
\usepackage{amscd}
\usepackage{amsthm}
\usepackage{latexsym}
\usepackage{amsbsy}
\usepackage{xcolor,enumerate}
\usepackage[normalem]{ulem}

\def\0D{\Delta^{(0)}}
\def\1D{\Delta^{(1)}}

\newcommand{\propA}{property \textbf{A} }

\newcommand{\Cc}{\mathcal{C}}
\newcommand{\Tc}{\mathcal{T}}
\newcommand{\Mc}{\mathcal{M}}
\newcommand{\Zc}{\mathcal{Z}}

\newcommand{\Oc}{\mathcal{O}}
\newcommand{\Ac}{\mathcal{A}}

\newcommand{\hmod}{{_H\mathcal{M}}}
\newcommand{\amod}{{_A\mathcal{M}}}
\newcommand{\Bmod}{{_B\mathcal{M}}}
\newcommand{\cmod}{{_C\mathcal{M}}}

\newcommand{\kmod}{{_K\mathcal{M}}}
\newcommand{\hhmod}{{_{H^2}\mathcal{M}}}

\newcommand{\Dc}{\mathcal{D}}
\newcommand{\Nc}{\mathcal{N}}

\newcommand{\Ec}{\mathcal{E}}
\newcommand{\aydc}{\widehat{a\mathcal{YD}}}
\newcommand{\aydm}{a\mathcal{YD}}
\newcommand{\indlimD}{\displaystyle{\varinjlim_{\Delta}}}
\newcommand{\projlimD}{\displaystyle{\varprojlim_{\Delta}}}

\newtheorem{theorem}{Theorem}[section]
\newtheorem{remark}[theorem]{Remark}
\newtheorem{proposition}[theorem]{Proposition}
\newtheorem{lemma}[theorem]{Lemma}
\newtheorem{corollary}[theorem]{Corollary}

\newtheorem{definition}[theorem]{Definition}

\def\build#1_#2^#3{\mathrel{\mathop{\kern 0pt#1}\limits_{#2}^{#3}}}

\newcommand{\vect}{\text{Vec}}

\newcommand{\bimod}{\mathbb{B}\text{md}}

\newcommand{\la}{\triangleright}
\newcommand{\ra}{\triangleleft}
\newcommand{\lba}{\blacktriangleright}

\newcommand{\bt}{\boxtimes}
\newcommand{\mmm}{{\text{-mod}}}

\numberwithin{equation}{section}

\def\ot{\otimes}
\def\part{\partial}

\def\ot{\otimes}

\def\build#1_#2^#3{\mathrel{
\mathop{\kern 0pt#1}\limits_{#2}^{#3}}}

\numberwithin{equation}{section}
\newcommand{\comment}[1]{\relax}

\begin{document}

\title{Categorified Chern character and cyclic cohomology.}
\author{Ilya Shapiro}

\date{}
\maketitle

\begin{abstract}
 We examine Hopf cyclic cohomology in the same context as the analysis \cite{it, loop1, loop2} of the geometry of loop spaces $LX$ in derived algebraic geometry and the resulting close relationship between $S^1$-equivariant quasi-coherent sheaves on $LX$ and $D_X$-modules.  Furthermore, the Hopf setting serves as a toy case for the categorification of Chern character theory as discussed in \cite{toenvezz1}. More precisely, this examination naturally leads to a  definition of mixed anti-Yetter-Drinfeld contramodules which reduces to that of the usual mixed complexes for the trivial Hopf algebra and generalizes the notion of stable anti-Yetter-Drinfeld contramodules that have thus far served as the coefficients for Hopf-cyclic theories \cite{contra}.  The cohomology is then obtained as a $Hom$ in this dg-category between a Chern character object associated to an algebra and an arbitrary coefficient mixed anti-Yetter-Drinfeld contramodule. 
\end{abstract}


\section{Introduction}

This paper is motivated by the connection, observed in \cite{hks}, between the Hopf-cyclic theory coefficients and the objects appearing in a twice categorified 1dTFT (one dimensional topological field theory) \cite{dsps}.  We begin to explore this link and its implications with the goal of better understanding and generalizing coefficients, as well as reinterpreting the Hopf cyclic cohomology itself.  The main concepts involved are as follows.  For a monoidal category $\Cc$ and its bimodule category $\Mc$ we have the Hochschild homology category $HH(\Cc,\Mc)$ and the cyclic homology category $HC(\Cc)$.  These are categorifications of the Hochschild homology of an algebra with coefficients in a bimodule, and the cyclic homology of an algebra respectively.  Furthermore, as the title suggests, the notion of a categorified version of the Chern character is essential for our purposes and is also of independent interest.

We do not    actually use the language of differential graded categories, instead everything takes place inside the $2$-category of categories with the additional requirement that functors possess right adjoints.  Often we consider the dual picture where the adjoints are themselves the $1$-morphisms and so we have a $2$-category of categories with functors possessing left adjoints.  The disadvantage of foregoing dg-categories is that the naive, i.e., resulting from the $2$-category setting above, $HC(\Cc)$ is not sufficient for our purposes; though it recovers the classical coefficients.  The advantage is that the exposition is considerably less sophisticated.  We get around the shortcomings of the naive approach in an ad hoc manner that is suggested by the dg setting.

Briefly, for a Hopf algebra $H$, motivated by some results in derived algebraic geometry \cite{it, loop1, loop2}, we propose a generalization of stable anti-Yetter-Drinfeld contramodules (with \emph{mixed} replacing stable) as an analogue of  $QC(LX)^{S^1}$, the $S^1$-equivariant quasi-coherent sheaves on the derived loop space of $X$.  This category serves both as the target for the categorified Chern characters of $H$-module algebras and also as the source of coefficients for the cohomology theory.  The Hopf-cyclic cohomology is then recovered as an $Ext$ in this category as was previously done by Connes and Kassel for cyclic cohomology, using cyclic objects \cite{connes}, and mixed complexes \cite{kassel} respectively.  This places Hopf-cyclic cohomology into the same framework as de Rham cohomology. Roughly speaking, $QC(LX)^{S^1}$ (related to $D_X$-modules)  is the cyclic homology of the monoidal category $QC(X)$, while \emph{mixed} anti-Yetter-Drinfeld contramodules form the cyclic homology (in the proper sense) of $\hmod$, the category of modules over the Hopf algebra $H$.

\textbf{Conventions}: All algebras $A$ in monoidal categories $\Cc$ are assumed to be unital associative, unless stated otherwise; we say that $A\in Alg(\Cc)$.  We let $A_\Cc\mmm$ stand for left $A$-modules in $\Cc$ which is a right $\Cc$-module category.  Similarly, $\bimod_\Cc A$ denotes the dual monoidal category of $A$-bimodules in $\Cc$. Our $H$ is a Hopf algebra over some ground field $k$ which is fixed throughout the paper, thus $\vect$ denotes the category of $k$-vector spaces.

\textbf{Outline}:  Section \ref{sec:limit} is provided as a guide and contains no proofs.  It uses the TFT formalism as motivation.  Section \ref{hochschildhomology} discusses, in general terms, the construction of a monad on a bimodule category such that modules over the monad are precisely its naive Hochschild homology category.  This construction is not always possible, but proceeds as described provided that, what we call, \propA  holds. Section \ref{sec:cyclicgen} describes what happens in the cyclic case.  Section \ref{ss:hhadj} expands on the discussion began in Section \ref{sec:limit} concerning the adjoint pair of functors relating the Hochschild homology categories of $\Cc$ and $\Dc$ and induced by an admissible $(\Cc,\Dc)$-bimodule $\Mc$.  This helps motivate the Definition \ref{def:chernhmod} of the (categorified) Chern character, and will also be used in Section  \ref{sec:appl} in the Hopf case.  

Section \ref{specialcase} applies the notions of Section \ref{hochschildhomology} to the case of $\hmod$, the category of modules over a Hopf algebra.  In particular it is shown that \propA holds for $\hmod$. Furthermore, $HH(\hmod)$, the naive Hochschild homology of $\hmod$ coincides with anti-Yetter-Drinfeld contramodules as seen from Proposition \ref{prop:ayd} and $HC(\hmod)$, the naive cyclic homology of $\hmod$ coincides with the stable  anti-Yetter-Drinfeld contramodules as seen from Theorem \ref{thm:stableayd}.

Section \ref{subsec:mixedaYDcontra} begins by addressing the deficiency of $HC(\Cc)$ in the Definition \ref{S1act}; the $HH(\Cc)^{S^1}$ so defined is certainly the correct cyclic homology of $\Cc$ in the dg-sense.   We proceed by defining the Chern character $ch(A)$ in an ad hoc manner, motivated by Sections \ref{sec:limit} and \ref{ss:hhadj}.  We conjecture that the definition would be a consequence of those sections, were they to be made rigorous.  In Section \ref{subsec:nvso} we prove Theorem \ref{th:oldnew} which relates the usual definition \cite{contra} of Hopf cyclic cohomology of an algebra  $A$ with coefficients in a stable  anti-Yetter-Drinfeld contramodule $M$ to an $Ext$ between $ch(A)$ and $M$ in the category of mixed anti-Yetter-Drinfeld contramodules.  The stable $M$ can of course be replaced by a mixed one, thereby generalizing the coefficients.  Alternatively, in Section \ref{tsygan} we pursue Tsygan's approach to cyclic homology via a double complex.  More precisely, we describe a tricomplex that computes Hopf-cyclic cohomology of $A$ with coefficients in a mixed anti-Yetter-Drinfeld contramodule.  The advantage is that it does not require $A$ to be unital.  We make no claims about the compatibility of this approach with the $Ext$ one.

Section \ref{sec:comodules} is a summary of what would happen to Section \ref{specialcase} were it repeated for $\Mc^H$, the monoidal category of $H$-comodules.  It is an outline without proofs, though these can easily be filled in, unlike in the case of the more conjectural sections.  The main point is that stable  anti-Yetter-Drinfeld \emph{modules} arise as $HC(\Mc^H)$.

Section \ref{sec:appl} contains a discussion of applications of Sections \ref{sec:limit} and \ref{ss:hhadj} to the Hopf setting. In particular \eqref{eq:theadj11}, obtained from the relationship between Hochschild homologies of $\hmod$ and $\Mc^H$ via the admissible bimodule category $\vect$, is exactly that which forms the backbone of \cite{s2}.

Finally, Section \ref{sec:app} is meant to provide a bridge between this paper and the literature that motivated it.  More precisely, in Section \ref{ss:qcx} we attempt to explain, by sketching the commutative case, why the mixed anti-Yetter-Drinfeld contramodules in the Hopf setting are exactly analogous to the $D$-modules in the geometric setting.  In Section \ref{contratraces} we link the monadic approach with the previous attempts \cite{hks, ks, ks2, s, s2} to understand Hopf cyclic coefficients as centers of certain bimodule categories.

\textbf{Acknowledgments}: The author wishes to thank Masoud Khalkhali and Ivan Kobyzev for stimulating discussions.   This research was supported in part by the NSERC Discovery Grant number 406709.

\section{TFT patterns}\label{sec:limit}

This section is motivational in nature and is concerned with exploring the TFT point of view that will be helpful in the analysis of Hopf cyclic cohomology.

Considering twice categorified local $1$-dimensional topological field theories, see \cite{dsps}  for a good overview, we obtain a way to keep track of all the conjectural relationships between the notions considered in the present paper.  Namely, such a TFT assigns a monoidal category $\Cc$ to a point and consequently $HH(\Cc)$ to $S^1$, and the non-trivial self-homotopy of the identity on $S^1$ is assigned an automorphism of the identity functor on $HH(\Cc)$, i.e., the $S^1$-action.   A map between TFTs $\Cc$ and $\Dc$ is then given by a suitable $(\Cc,\Dc)$-bimodule category $\Mc$.

More precisely, as suggested by \cite{s}, we consider only admissible $\Mc$, i.e.,
\begin{definition}\label{def:admissible}
We say that a $(\Cc,\Dc)$-bimodule category $\Mc$ is admissible if $\Mc\simeq B_\Dc\mmm$, for some algebra $B\in\Dc$, as a right $\Dc$-module category, and the left $\Cc$-action is given by a monoidal functor $\Cc\to\bimod_\Dc B$ that has a right adjoint.
\end{definition}

Furthermore, we expect that (see Section \ref{ss:hhadj} for constructions)  such an $\Mc$ induces an adjoint pair of functors $(HH(\Mc)^*,HH(\Mc)_*)$: \begin{equation}\label{eq:hhadj}HH(\Mc)^*:HH(\Cc)\leftrightarrows HH(\Dc):HH(\Mc)_*\end{equation} that are compatible with the $S^1$-actions.  Given an $A\in Alg(\Cc)$, we observe that $A_\Cc\text{-mod}$ is an example of an admissible $(\vect,\Cc)$-bimodule, thus: 

\begin{definition}\label{def:chern1}

Let the Chern character of $A$ be \begin{equation*}ch(A)=HH(A_\Cc\text{-mod})^*k\in HH(\Cc)^{S^1}\end{equation*} where $k\in\vect^{S^1}$ is the trivial mixed complex. 
\end{definition}

Moreover, given an admissible $(\Dc,\Ec)$-bimodule $\Nc$, it is possible (\cite{s} and Remark \ref{rem:t}) to define an admissible $(\Cc,\Ec)$-bimodule suggestively denoted $\Mc\bt_\Dc \Nc$, and we should have an equivalence of $S^1$-equivariant functors: $$HH(\Mc\bt_\Dc\Nc)^*\simeq HH(\Nc)^*HH(\Mc)^*$$ and thus $HH(\Mc\bt_\Dc\Nc)_*\simeq HH(\Mc)_*HH(\Nc)_*$. 

\begin{definition}
Given an $A\in Alg(\Cc)$, let $\Mc^*A\in Alg(\Dc)$ be the image of $A$ under $\Cc\stackrel{act}{\to}\bimod_\Dc B\stackrel{U}{\to}\Dc$, where $U$ forgets the $B$-bimodule structure.  

\end{definition}

We have $A_\Cc\mmm\bt_\Cc\Mc=(\Mc^*A)_\Dc\mmm$ and \begin{equation}\label{eq:comp}HH(\Mc)^*ch(A)\simeq ch(\Mc^*A)\end{equation} since the former is $$HH(\Mc)^*HH(A_\Cc\text{-mod})^*k\simeq HH(A_\Cc\mmm\bt_\Cc\Mc)^*k=HH((\Mc^*A)_\Dc\mmm)^*k$$ which is the latter.

\begin{remark}\label{rem:t}
Note that there is of course a more conceptual definition of $\Mc\bt_\Dc\Nc$ as the Hochschild homology category, i.e., $HH(\Dc,\Mc\bt\Nc)$.   By \cite{it} the concrete description above and the conceptual one agree.   More precisely, we have $A_\Cc\text{-mod}\bt_\Cc\Mc=A_\Mc\text{-mod}$, where the former has the Hochschild homology definition.
\end{remark}

\subsection{A note on the usual Chern character}
The description of the Chern character in Definition \ref{def:chern1} is a categorification of the usual one.  On the other hand it ``extends down".  More precisely, we can continue the TFT patterns from above to the consideration of what happens to a suitable $F\in Fun_\Cc(\Mc,\Nc)_\Dc$ where $\Mc$ and $\Nc$ are admissible $(\Cc,\Dc)$-bimodules.  It should yield an associated natural transformation $$HH(F)^*:HH(\Nc)^*\to HH(\Mc)^*.$$

In the case of a suitable $M\in\bimod_\Cc(A,B)$ for $A,B\in Alg(\Cc)$ and the resulting $M\ot_B-:B_\Cc\mmm\to A_\Cc\mmm$ we would get $$HH(M\ot_B-)^*:HH(A_\Cc\mmm)^*\to HH(B_\Cc\mmm)^*$$ and so the usual Chern character of $M$, i.e., $$ch(M)\in RHom^0_{HH(\Cc)^{S^1}}(ch(A),ch(B))$$ where $ch(M)=HH(M\ot_B-)^*k$.  Note that if $N$ is a suitable right $B$-module in $\Cc$ then $ch(N)\in RHom^0_{HH(\Cc)^{S^1}}(ch(1),ch(B))=Ext^0_{mixed}(k, u_*ch(B))$, with $u_*$ the right adjoint of $u^*:\vect\to\Cc$ (unit inclusion).  We remark that despite appearing to, at first glance, be concentrated in degree $0$, the Chern character  $ch(N)$ is spread out as expected (see Remark \ref{ch1comp}). We will not address the above in this paper as it is not relevant for our present purposes.

\subsection{Does $\Mc_* E$ exist?}\label{ss:lower}

Above we saw that given $A\in Alg(\Cc)$ we have a corresponding $\Mc^* A\in Alg(\Dc)$ for $\Mc$ an admissible $(\Cc,\Dc)$-bimodule category.  We can view $-\bt_\Cc \Mc$ as a $2$-functor from right $\Cc$-modules, i.e., $\text{Mod}\Cc$,  to $\text{Mod}\Dc$ that preserves admissibility.  As such it has a right adjoint $Fun(\Mc,-)_\Dc$.  Observe that for $E\in Alg(\Dc)$ we have $Fun(\Mc,E_\Dc\mmm)_\Dc\simeq\bimod_\Dc(E,B)$ with the right $\Cc$-action via $\Cc\to\bimod_\Dc B$.  Unfortunately defining $\Mc_* E$ to satisfy $(\Mc_* E)_\Cc\mmm\simeq Fun(\Mc,E_\Dc\mmm)_\Dc$ need not work as such an algebra need not exist in $\Cc$.

We can see the problem more clearly in the special case of $\Mc=\Dc$, i.e., we have a monoidal $f^*:\Cc\to\Dc$ with a right adjoint $f_*$. In this case $Fun(\Mc,-)_\Dc$ is simply the pullback of right $\Dc$-modules to right $\Cc$-modules via $f^*$.  Now for $A\in Alg(\Cc)$ the corresponding algebra in $\Dc$ is $f^*(A)$ and furthermore, we observe that for $E\in Alg(\Dc)$ the object $f_* (E)$ also has an algebra structure via $f_*(E)\ot f_*(E)\to f_*(E\ot E)\to f_*(E)$.  What we would like is to have an equivalence $\bimod_\Dc(E,f^*(A))\simeq\bimod_\Cc(f_*(E),A)$.

Let us take $\Cc=\hmod$ (the category of modules over a Hopf algebra), $\Dc=\vect$ (the category of vector spaces), and  $f^*$ the fiber functor. Take $A=1$ and $E=1$, then we have $\bimod_\vect(1,1)=\vect$ while $\bimod_\hmod(H^*,1)=H^*_\hmod\mmm$.  They are almost equivalent, i.e., if $H$ is finite dimensional, but not otherwise.   Another example is considered in Section \ref{ss:qcx}.

On the other hand there are examples where it is possible to define $\Mc_*E$, but it is not $f_*E$.  Namely, take $\epsilon^*:\vect\to\hmod$ to be the monoidal unit, with right adjoint $V\mapsto V^H$.  We observe that for any $E\in Alg(\hmod)$ the right $\hmod$-module $E_\hmod\mmm$ is equivalent, as a right $\vect$-module to $E\rtimes H\mmm$, where $(a\ot x)(b\ot y)=a(x^1b)\ot x^2y$, and not in general to $E^H\mmm$; thus \begin{equation}\label{eq:push}\epsilon_*E=E\rtimes H.\end{equation}  This case considers $\hmod$ as an admissible $(\vect,\hmod)$-bimodule category. Note that we can also consider it as an admissible $(\hmod,\vect)$-bimodule category since $\hmod=H_\vect\mmm$ and the resulting map $\hmod\to\bimod H$ sending $V$ to $V\ot H$ with $x(v\ot h)y=x^1v\ot x^2hy$ has a right adjoint $S\mapsto ad(S)$ where the action on the latter is via $h\cdot s=h^1sS(h^2)$.

Recall that $HH(\Mc\bt_\Dc\Nc)^*\simeq HH(\Nc)^*HH(\Mc)^*$, i.e., compatibility of $HH(-)$ with composition, implies that $HH(\Mc)^*ch(A)\simeq ch(\Mc^*A)$ so that $ch(-)$ is compatible with pullbacks.  If we assume the compatibility of $HH(-)$ with internal Homs when they exist, i.e., $$HH(Fun(\Nc,\Tc)_\Ec)^*\simeq HH(\Nc)_*HH(\Tc)^*$$ then we obtain the compatibility of $ch(-)$ with pushforwards.  More precisely, defining $\Mc_*E$ by $(\Mc_*E)_\Cc\mmm=Fun(\Mc,E_\Dc\mmm)_\Dc$, if possible, we have  $ch(\Mc_*E)=HH(Fun(\Mc,E_\Dc\mmm)_\Dc)^*k$ while the latter is equivalent to $HH(\Mc)_*HH(E_\Dc\mmm)^*k$ and so to $HH(\Mc)_*ch(E)$, i.e., \begin{equation}\label{eq:comp2}HH(\Mc)_*ch(E)\simeq ch(\Mc_*E).\end{equation}

\section{Hochschild and cyclic homology categories}\label{hochschildhomology}
Here we concern ourselves with understanding the naive Hochschild and cyclic homologies of monoidal categories.  We make certain additional assumptions, satisfied in our intended setting, that allow for a description of Hochschild homology as the category of modules over a monad.  In the cyclic case the monad is equipped with a central element that supplies the $S^1$-action.

Let $\Mc$ be a bimodule category over $\Cc$ such that both actions: $\Delta^*_l:\Cc\bt\Mc\to\Mc$ and $\Delta^*_r:\Mc\bt\Cc\to\Mc$  possess right adjoints: $\Delta_{l*}$ and $\Delta_{r*}$ respectively. We also require that the unit and the product functors in $\Cc$ possess right adjoints as well.  

\begin{definition}
We will say that a $\Cc$-bimodule $\Mc$, as above, has \emph{right adjoints}.
\end{definition}

Let us begin with  its Hochschild homology category $HH(\Cc,\Mc)$.  In the usual way, using the left and right actions together with the unit and product of $\Cc$ we may view $\Mc\boxtimes\Cc^{\boxtimes\bullet}$ as a simplicial object  in the $2$-category of categories \emph{with functors possessing right adjoints}.  All of the subtleties of tensor products of categories themselves are avoided \emph{in our case} as they all arise as categories of modules over some algebra, see Section \ref{specialcase}.

\begin{remark}\label{r:aydmod}
We can also consider the setting of categories with functors possessing left adjoints.  This can arise in the situation that we consider in Section \ref{specialcase}. Namely, given a map of algebras $f:A\to B$ the forgetful functor $f^*$ from $B$-modules to $A$-modules has both a right, $f_*=Coinduction$,  and a left, $f_!=Induction$,  adjoint.  We mention that the former case is what we consider, and it leads to anti-Yetter-Drinfeld \emph{contramodules} via a certain monad.  The latter would lead to the case of anti-Yetter-Drinfeld \emph{modules}  through the consideration of the co-monad which is the left adjoint to our monad that we will call $\Ac$ below.
\end{remark}

Let $\Delta$ denote the simplex category (finite linearly ordered sets with order preserving maps), then $$HH(\Cc,\Mc)=\indlimD\,\Mc\boxtimes\Cc^{\boxtimes\bullet},$$ however we may use the right adjoints of the above structure functors to view $\Mc\boxtimes\Cc^{\boxtimes\bullet}$ as a co-simplicial object in categories with functors possessing \emph{left adjoints} so that we have $$HH(\Cc,\Mc)=\projlimD\,\Mc\boxtimes\Cc^{\boxtimes\bullet};$$the latter is simpler to describe explicitly.

Generally if we have $\Cc_\bullet$ ``fibered" over some index category $X$,  then the inverse limit consists of the following data: for every $x\in X$ we are given an $M_x\in \Cc_x$ together with isomorphisms  $\tau_f:f(M_x)\to M_y$ for any $f:x\to y$.  Furthermore, if $g:y\to z$ then $\tau_g g(\tau_f)=\tau_{gf}$ and $\tau_{Id_x}=Id_{M_x}$.

This data in the case of the simplex index category reduces to an $M\in\Cc_0$ together with an isomorphism (the $\delta$'s and $s$'s below are the usual faces and degeneracies) in $\Cc_1$: $$\tau: \delta_0(M)\to\delta_1(M),$$  subject to the unit condition in $\Cc_0$: $$s_0(\tau)=Id_M,$$ and the associativity condition in $\Cc_2$: \begin{equation}\label{associativity}\delta_2(\tau)\delta_0(\tau)=\delta_1(\tau).\end{equation}  Naturally, $M_i\in\Cc_i$ is obtained from $M$ via $M_i=\delta_{i}\delta_{(i-1)}\cdots \delta_{1}M$, i.e., $M_i=low(M)$ where $low:[0]\to[i]$ maps the point to the smallest element.

\subsection{The \propA assumption}
In order to realize the Hochschild homology category as the category of modules over a monad we require additional assumptions that necessitate the Definition \ref{def:base}.  More precisely, we wish to describe the inverse limit (that gives $HH(\Cc,\Mc)$ in our intended setting) as the category of modules over a monad on $\Mc$. 

Recall the general setting of $\Cc_\bullet$, a co-simplicial object in categories (with left adjoints). We will denote the faces adjoint pairs by $(\delta^*,\delta_*)$ and the degeneracies by $(s^*,s_*)$.  

We have $\delta_1^*\delta_0^*=\delta_0^*\delta_2^*:\Cc_2\to\Cc_0$ and so obtain by adjunction (see Section \ref{basechangeforalgebras}):\begin{equation}\label{basechange1}\delta_2^*\delta_{0*}\to\delta_{0*}\delta_1^*.\end{equation} While $\delta_1^*\delta_1^*=\delta_1^*\delta_2^*:\Cc_2\to\Cc_0$ yields:\begin{equation}\label{basechange2}\delta_2^*\delta_{1*}\to\delta_{1*}\delta_1^*.\end{equation}  Furthermore, from $\delta^*_0\delta^*_3=\delta^*_2\delta^*_0:\Cc_3\to\Cc_1$ we have: \begin{equation}\label{basechange3}\delta^*_3\delta_{0*}\to\delta_{0*}\delta^*_2.\end{equation}

\begin{definition}\label{def:base}
We say that $\Cc_\bullet$ has \propA if  \eqref{basechange1},  \eqref{basechange2}, and \eqref{basechange3} are isomorphisms.  We say that a $\Cc$-bimodule $\Mc$, with right adjoints, has \propA if $\Cc_\bullet=\Mc\bt\Cc^{\bt\bullet}$ has \propA\!\!.
\end{definition}

\begin{definition}
Let $\Ac=\delta^*_1\delta_{0*}\in End(\Cc_0)$.  This is our monad.
\end{definition}

Define the multiplication on $\Ac$ as follows.  First observe that we have a composition of isomorphisms: \begin{equation}\label{taum}\tau_m:\delta_{0*}\delta^*_1\delta_{0*}\leftarrow\delta^*_2\delta_{0*}\delta_{0*}=
\delta^*_2\delta_{1*}\delta_{0*}\rightarrow \delta_{1*}\delta^*_1\delta_{0*},\end{equation} where the first arrow is provided by \eqref{basechange1} (and the \propA assumption), the second equality is due to $\delta_{0*}\delta_{0*}=\delta_{1*}\delta_{0*}$, and the third arrow is \eqref{basechange2}.  The multiplication $m:\Ac^2\to\Ac$ is obtained from $\tau_m$ by adjunction.

\begin{lemma}\label{ass}
Let $M\in\Cc_0$ with $\tau:\delta_{0*}M\to\delta_{1*}M$. Equip $M$ with $a:\Ac(M)\to M$ via adjunction from $\tau$. 
Then it is an action if and only if $\delta_{2*}(\tau)\delta_{0*}(\tau)=\delta_{1*}(\tau)$.
\end{lemma}

\begin{proof}
Note that  $$\xymatrix{\delta^*_1\delta_{0*}\delta^*_1\delta_{0*}M\ar[rr]^{\delta^*_1\delta_{0*}(a)}\ar[d]^m & & \delta^*_1\delta_{0*}M\ar[d]^a\\
\delta^*_1\delta_{0*}M\ar[rr]^a & & M}$$ commutes if and only if $$\xymatrix{\delta_{0*}\delta^*_1\delta_{0*}M\ar[rr]^{\delta_{0*}(a)}\ar[d]^{\tau_m} & & \delta_{0*}M\ar[d]^\tau\\
\delta_{1*}\delta^*_1\delta_{0*}M\ar[rr]^{\delta_{1*}(a)} & & \delta_{1*}M}$$ commutes.  We can expand the vertical $\tau_m$ of the latter diagram into $$\xymatrix{\delta^*_2\delta_{0*}\delta_{0*} M\ar[r]\ar[d]^= & \delta_{0*}\delta^*_1\delta_{0*}M\ar[rr]^{\delta_{0*}(a)} & & \delta_{0*}M\ar[d]^\tau\\
\delta^*_2\delta_{1*}\delta_{0*}M\ar[r] & \delta_{1*}\delta^*_1\delta_{0*}M\ar[rr]^{\delta_{1*}(a)} & & \delta_{1*}M}$$ which commutes if and only if $$\xymatrix{\delta_{0*}\delta_{0*} M\ar[r]\ar[d]^= & \delta_{2*}\delta_{0*}\delta^*_1\delta_{0*}M\ar[rr]^{\delta_{2*}\delta_{0*}(a)} & & \delta_{2*}\delta_{0*}M\ar[d]^{\delta_{2*}(\tau)}\\
\delta_{1*}\delta_{0*}M\ar[r] & \delta_{2*}\delta_{1*}\delta^*_1\delta_{0*}M\ar[rr]^{\delta_{2*}\delta_{1*}(a)} & & \delta_{2*}\delta_{1*}M}$$ commutes.  Observe that under the identification $\delta_{2*}\delta_{0*}=\delta_{0*}\delta_{1*}$ the top row becomes $\delta_{0*}\delta_{0*} M\to\delta_{0*}\delta_{1*}\delta^*_1\delta_{0*}M\stackrel{\delta_{0*}\delta_{1*}(a)}{\longrightarrow}\delta_{0*}\delta_{1*}M$, i.e., $\delta_{0*}(\delta_{0*} M\to\delta_{1*}\delta^*_1\delta_{0*}M\stackrel{\delta_{1*}(a)}{\longrightarrow}\delta_{1*}M)$, which is $\delta_{0*}(\delta_{0*}M\stackrel{\tau}{\to}\delta_{1*}M)$, so $\delta_{0*}(\tau)$.  Similarly, under the identification $\delta_{2*}\delta_{1*}=\delta_{1*}\delta_{1*}$ the bottom row becomes $\delta_{1*}(\tau)$ and we are done.

\end{proof}

Consider $u:Id\to\Ac$ given by \begin{equation}\label{theu}u:Id=\delta_1^*s_0^*s_{0*}\delta_{0*}\to\delta_1^*\delta_{0*}\end{equation} using the evaluation $s_0^*s_{0*}\to Id$. We have the following easy lemma (with proof as in Lemma \ref{ass}).

\begin{lemma}\label{unit}
Let $M\in\Cc_0$ with $\tau:\delta_{0*}M\to\delta_{1*}M$.  Then $s_{*0}(\tau)=Id$ if and only if $M\stackrel{u}{\to}\Ac(M)\stackrel{a}{\to}M$ is $Id$.
\end{lemma}

\subsection{Associativity and unitality of $\Ac$}
In order to demonstrate that $\Ac$ is  unital and associative it is helpful to consider cospans in the simplex category.  The idea is that $\Ac$ is actually defined (is a $1$-morphism from $[0]$ to itself) in the $2$-category $2\Delta$ obtained from $\Delta$ by freely adding left adjoints see \cite{freeadd} for example.  In particular, when pushouts are available we can compose cospans and obtain $2$-morphisms in  $2\Delta$ via maps of cospans.  We observe that for $n\geq 1$ we have a pushout square (we will need $n=1,2$): $$\xymatrix{[n-1]\ar[d]^{\delta_n}\ar[r]^{\delta_0}&[n]\ar[d]^{\delta_{n+1}}\\
[n]\ar[r]^{\delta_0}& [n+1]}$$ which allows us to compose as follows:$$\xymatrix{
[0]\ar[dr]^{\delta_0}&& [0]\ar[dr]^{\delta_0}\ar[dl]^{\delta_1}&& [0]\ar[dr]^{\delta_0}\ar[dl]^{\delta_1}&& [0]\ar[dl]^{\delta_1}\\
&[1]\ar[dr]^{\delta_0} && [1]\ar[dr]^{\delta_0}\ar[dl]^{\delta_2} && [1]\ar[dl]^{\delta_2} &\\
&& [2]\ar[dr]^{\delta_0} && [2]\ar[dl]^{\delta_3} &&\\
&&& [3] &&&
}$$ so that when considering $$\Ac=\delta^*_1\delta_{0*}=[0]\stackrel{\delta_0}{\to}[1]\stackrel{\delta_1}{\leftarrow}[0]$$ we have that $$\Ac^2=[0]\stackrel{high}{\to}[2]\stackrel{low}{\leftarrow}[0]$$ and the product $m:\Ac^2\to\Ac$ is given by the unique map of the cospans $\delta_1:[1]\to[2]$, where one should note the \emph{reversal} of the arrows specifying the $2$-morphisms. 

\begin{lemma}
The $1$-loop $\Ac$ at $[0]$ in $2\Delta$ is associative.
\end{lemma}
\begin{proof}
The two $2$-morphisms $m\circ Id$ and $Id\circ m$ from $\Ac^3$ to $\Ac$ coincide since there is only a unique map from $[0]\stackrel{high}{\to}[3]\stackrel{low}{\leftarrow}[0]$ to $[0]\stackrel{\delta_0}{\to}[1]\stackrel{\delta_1}{\leftarrow}[0]$, i.e., $f:[1]\to [3]$ with $(0,1)\mapsto (0,3)$.
\end{proof}

Similarly, the unit $u$ is given by $s_0:[1]\to [0]$ from identity represented as $[0]=[0]=[0]$ to $\Ac$.

\begin{lemma}
The map $u$ is the unit of $\Ac$ in $2\Delta$.
\end{lemma}
\begin{proof}
Checking that it is a left and a right unit is checking that two cospan maps from $[0]\stackrel{\delta_0}{\to}[1]\stackrel{\delta_1}{\leftarrow}[0]$ to itself are both identity, yet there is no other map.

\end{proof}

\begin{corollary}\label{ismonad}
Assuming that \eqref{basechange1} and \eqref{basechange3} are isomorphisms, the endofunctor $\Ac$ of $\Cc_0$ is a monad.

\end{corollary}
\begin{proof}
The assumptions allow us to interpret $\Ac^2$ and $\Ac^3$ as cospans in $\Cc_\bullet$.
\end{proof}

\begin{remark}
Note that \eqref{basechange1} is required in order to define $m$, while \eqref{basechange2} is needed to make sure that $\tau_m$ of \eqref{taum} is an isomorphism.  Finally, \eqref{basechange3} facilitates the proof of associativity. 
\end{remark}

\begin{theorem}\label{theorem:monadic}
Let $\Cc_\bullet$ denote a cosimplicial object in categories with functors possessing left adjoints.  Assume that $\Cc_\bullet$ has \propA\!\!, then $$HH(\Cc_\bullet)=\projlimD\,\Cc_\bullet=\Ac_{\Cc_0}\text{-mod}.$$
\end{theorem}

\begin{proof}
By Corollary \ref{ismonad} we have that $\Ac\in End(\Cc_0)$ is a monad.  Recall that the inverse limit consists of $M\in\Cc_0$ with isomorphisms $\tau_M:\delta_{0*}M\to\delta_{1*}M$ such that $\delta_{2*}(\tau_M)\delta_{0*}(\tau_M)=\delta_{1*}(\tau_M)$ (associativity) and $s_{0*}(\tau_M)=Id_M$ (unitality).  Furthermore, maps between such objects $M$ and $N$ are $\varphi\in Hom_{\Cc_0}(M,N)$ such that $\delta_{1*}(\varphi)\tau_M=\tau_N\delta_{0*}(\varphi)$.

The correspondence between $(M,\tau_M)$ and $(M,a_M)$, where $\tau_M:\delta_{0*}M\to\delta_{1*}M$ (inverse limit) and $a_M:\delta^*_1\delta_{0*}M\to M$ (action) is via adjunction.  Since $\delta_{1*}(\varphi)\tau_M=\tau_N\delta_{0*}(\varphi)$ if and only if $\varphi a_M=a_N\delta^*_1\delta_{0*}(\varphi)$, so the maps in the inverse limit coincide with the maps of $\Ac$-modules.  Note that by Lemmas \ref{ass} and \ref{unit} the associativity and unitality conditions coincide as well.

The only issue is that $\tau_M$ is assumed to be an isomorphism and it is not yet obvious that any such $\tau$ obtained from an action $a_M:\Ac(M)\to M$ via adjunction is automatically an isomorphism.  To that end observe that $\cdots\to\Ac^2(M)\to\Ac(M)\to M\to 0$ is a complex of $\Ac$-modules that is null-homotopic in $\Cc_0$ and thus we have a commutative diagram of exact sequences: $$\xymatrix{
\delta_{0*}\Ac^2(M)\ar[r]\ar[d]^{\tau_m(\Ac(M))} & \delta_{0*}\Ac(M)\ar[r]\ar[d]^{\tau_m(M)} & \delta_{0*}M\ar[d]^{\tau_M}\ar[r] & 0\\
\delta_{1*}\Ac^2(M)\ar[r] & \delta_{1*}\Ac(M)\ar[r] & \delta_{1*}M\ar[r] & 0\\
}$$ and recall that $\tau_m$ (see \eqref{taum}) is an isomorphism (this is where we used \eqref{basechange2} being an isomorphism) and so $\tau_M$ is as well. 
\end{proof}

\begin{proposition}\label{prop:adj}
Let $\Mc$ be a $\Cc$-bimodule category with right adjoints and \propA\!\!.  Then $$HH(\Cc,\Mc)=\Ac_\Mc\text{-mod}$$ and the canonical map $Tr:\Mc\to HH(\Cc,\Mc)$ is realized as part of an adjoint pair of functors $(F,U)$: $$F:\Mc\leftrightarrows \Ac_\Mc\text{-mod}: U$$ where $U$ forgets the $\Ac$-module structure and $F$ freely creates it via $F(M)=\Ac(M)$.
\end{proposition}

\begin{proof}
This is immediate from Theorem \ref{theorem:monadic} applied to $\Mc\bt\Cc^{\bt\bullet}$.  In particular, it is clear that the right adjoint of $Tr$ coincides with $U$.
\end{proof}

\begin{remark}\label{rem:trace}
We observe that for formal reasons there is a canonical identification $\tau_{C,M}:F(C\cdot M)\simeq F(M\cdot C)$ in $\Ac_\Mc\text{-mod}$  for any $C\in\Cc$ and $M\in\Mc$.  More generally, for $T\in\Mc\bt\Cc$ we have $\tau_T:F(\delta_1^* T)\simeq F(\delta_0^* T)$.  This identification is unital and associative in the sense of Lemma \ref{contracoh}, for the same reasons as in its proof.
\end{remark}

\subsection{Cyclic homology of a monoidal category $\Cc$}\label{sec:cyclicgen}
If in the discussion of Section \ref{hochschildhomology} we replace an arbitrary $\Cc$-bimodule category $\Mc$ by $\Cc$ itself, then the structure on $\Cc^{\boxtimes \bullet+1}$ of a cosimplicial object in categories (with left adjoints) extends to that of a cocyclic object in the same $2$-category. In this case the inverse limit of that diagram is the cyclic homology  $HC(\Cc)$ of $\Cc$.

In this section we will address the more general version of the above where we consider $\Cc_\bullet$, a co-cyclic object in categories (with left adjoints).  We will require that when considered as a co-simplicial object, it satisfy the \propA assumption. This will allow us to apply Theorem \ref{theorem:monadic} to it.  

\begin{definition}
Let $\Cc_\bullet$ be a co-cyclic object in categories (with left adjoints).  We say that it has \propA if its restriction to a co-simplicial object has \propA\!\!. 
\end{definition}

An important, though unsurprising,  new feature of $\Ac$ in this setting is $\varsigma: Id\to\Ac$ obtained from the extra degeneracy $s_1:[1]\to [0]$.  More precisely, \begin{equation}\label{thesigma}\varsigma:Id=\delta_1^*s_1^*s_{1*}\delta_{0*}\to \delta_1^*\delta_{0*}=\Ac.\end{equation}  Let $HH(\Cc_\bullet)$ denote the inverse limit of $\Cc_\bullet$ over the simplex category $\Delta$.  Recall  that it consists of $M\in\Cc_0$ equipped with an isomorphism $\tau_M:\delta_{0*}M\to\delta_{1*}M$ (that satisfies unitality and associativity).  Note that if $HC(\Cc_\bullet)$ is the inverse limit of $\Cc_\bullet$ over the Connes' cyclic category $\Lambda$, then the only additional requirement on $M$ is that \begin{equation}\label{extradeg}s_{1*}(\tau)=Id_M.\end{equation}  

On the other hand, without insisting on \eqref{extradeg}, we obtain that $HH(\Cc_\bullet)$ is equipped with an element $\varsigma=s_{1*}(\tau)$ of its Drinfeld center.  The latter assertion follows immediately from noting that if $\alpha:M\to N$ is a map in $HH(\Cc_\bullet)$ then $\delta_{1*}(\alpha)\tau_M=\tau_N\delta_{0*}(\alpha)$ and so applying $s_{1*}$ to everything we get $\alpha\varsigma_M=\varsigma_N\alpha$.  Note that since $\tau$ is an isomorphism, so is $s_{1*}(\tau)$, thus $\varsigma$ is an invertible element of the Drinfeld center of $HH(\Cc_\bullet)$.

\begin{remark}\label{rem:cyctr}
Recall Remark \ref{rem:trace} where it was pointed out that $Tr:\Mc\to HH(\Cc,\Mc)$ has a unital associative $\tau_T:Tr(\delta^*_1 T)\to Tr(\delta^*_0 T)$ for $T\in\Mc\bt\Cc$.  Observe that in our current setting of $\Mc=\Cc$, we have $Tr:\Cc \to HH(\Cc,\Cc)$ as well, and whereas the unitality meant: for $X\in\Cc$ we have $\tau_{s^*_0 X}=Id_{TrX}$, now we also get that \begin{equation}\label{eq:varsig}\tau_{s^*_1 X}=\varsigma_{TrX}. 
\end{equation} This implies, in particular, that for $X,Y\in\Cc$ we have $\tau_{Y,X}\tau_{X,Y}=\varsigma_{Tr(X\ot Y)}$.
\end{remark}

Observe that the  usage of $\varsigma$ for both the element of $\Ac$ in \eqref{thesigma} and the element of the Drinfeld center is not problematic, since it is immediate that they are one and the same. Namely, if $a:\Ac(M)\to M$ denotes the action incarnation of the structure on $M$ then  $s_{1*}(\tau)=a\varsigma(M)$.  Consequently, completely formally \eqref{thesigma} specifies a central element of $\Ac$.  To summarize:

\begin{lemma}
Let $\Cc_\bullet$ be a co-cyclic object in categories (with left adjoints) and \propA\!\!, then the monad $\Ac$ on $\Cc_0$ is equipped with a central element $\varsigma: Id\to \Ac$.  Furthermore  $HC(\Cc_\bullet)$ consists of those $\Ac$-modules in $\Cc_0$ on which $\varsigma$ acts by identity. 
\end{lemma}

\begin{corollary}\label{cor:cycliccoh1}
Let $\Cc$ be a monoidal category with right adjoints and \propA\!\!, then $HC(\Cc)=\Ac'_\Cc\text{-mod}$, i.e., the full subcategory of $HH(\Cc,\Cc)$  that consists of objects on which $\varsigma$ acts by identity.
\end{corollary}

Motivated by the $\infty$-category setting and the derived yoga, we find the following heuristic useful.  Given an algebra $A$ with a central element $\varsigma$,  to consider $A$-modules on which $\varsigma$ acts by identity is to consider $A/(1-\varsigma)$-modules. It would be more enlightened however (in the derived sense) to instead view the cone of $(1-\varsigma): A\to A$ as a differential graded algebra, i.e., $(A[\epsilon],d)$ where $deg(\epsilon)=-1$ and it too is central, and $d=(1-\varsigma)\iota_\epsilon$ (with $\iota$ denoting contraction). 

Note that we have an obvious map of differential graded algebras (DGAs) $i:A\to A[\epsilon]$ so that we get an adjoint pair of functors between $A$-modules (since $A$ is viewed as a DGA they are actually complexes) and $A[\epsilon]$-modules, namely induction and restriction.  Furthermore, we can explicitly describe the category of $A[\epsilon]$-modules as complexes $(M^\bullet,d,h)$ where $(M^\bullet,d)$ is a cohomological complex of $A$-modules, $(M^\bullet,h)$ is a homological complex of $A$-modules, and $dh+hd=1-\varsigma$.   Of course the cohomology of such $M^\bullet$ is a graded $A/(1-\varsigma)$-module.  The adjoint pair of functors is then seen to be the cone of $1-\varsigma$ (which acquires the homotopy $h$ automatically) and the functor that forgets the homotopy.

We can apply the above heuristic to the monad $\Ac$ and its $\varsigma$. We do this in section \ref{subsec:mixedaYDcontra}.  The result is the \emph{correct} cyclic homology category of $\Cc$.  More precisely, given $\Cc$ as above we replace the \emph{naive} $HC(\Cc)$ by complexes $(M^\bullet,d,h)$ in $HH(\Cc,\Cc)$ with $dh+hd=1-\varsigma$, where we may recall that $HH(\Cc,\Cc)$ comes equipped with an element $\varsigma$ of its Drinfeld center.

\section{Induced adjunctions between Hochschild homologies}\label{ss:hhadj}
We address the adjunction of \eqref{eq:hhadj} in more detail here.  By the Definition \ref{def:admissible} of an admissible bimodule $\Mc$ we need to deal with two dissimilar cases.  The first is the case of $\rho^*:\Cc\to\Dc$ a tensor functor with a right adjoint $\rho_*$.  The second is the case of $B\in Alg(\Dc)$ and the two monoidal categories $\Dc$ and $\bimod_\Dc B$.  Then we take the $\Dc$ in the first case to be the $\bimod_\Dc B$ of the second to obtain the adjunctions between $HH(\Cc)$ and $HH(\Dc)$ as required. Applications of this to the Hopf setting will be examined in Section \ref{sec:appl}.

\subsection{The monoidal functor}\label{sss:monfun}
Consider a monoidal functor $\rho^*:\Cc\to\Dc$ with a right adjoint $\rho_*$.  We will relate the Hochschild homologies of $\Cc$ and $\Dc$ by an adjoint pair of functors.  Observe that we have \begin{align*}\rho^*\Delta_\Cc^*\sigma\Delta_{\Cc*}\rho_*&\simeq\Delta^*_\Dc(\rho^*\bt\rho^*)\sigma(\rho_*\bt\rho_*)\Delta_{\Dc*}\\
&\simeq\Delta^*_\Dc\sigma(\rho^*\bt\rho^*)(\rho_*\bt\rho_*)\Delta_{\Dc*}\to\Delta_\Dc^*\sigma\Delta_{\Dc*}
\end{align*} and so, by adjunction, $\Ac_\Cc\rho_*\to\rho_*\Ac_\Dc$.  Thus we can define \begin{equation}\label{map3}HH(\rho)_*:\Ac_\Dc\mmm\to\Ac_\Cc\mmm\end{equation} as $T\mapsto\rho_*T$ with the action: $\Ac_\Cc\rho_*T\to\rho_*\Ac_\Dc T\to\rho_* T$.

On the other hand, also by adjunction, we have  \begin{equation}\label{map2}\rho^*\Ac_\Cc\to\Ac_\Dc\rho^*.\end{equation} We are now ready to construct the left adjoint to \eqref{map3}, i.e., for $T\in\Ac_\Cc\mmm$ the object $HH(\rho)^* T$ is defined to be the coequalizer of the diagram: $$\xymatrix{\Ac_\Dc\rho^*\Ac_\Cc T\ar[dr]_{Id_{\Ac_\Dc}\circ\eqref{map2}\,\,}\ar[rr]^{Id_{\Ac_\Dc\rho^*}\circ act} & &\Ac_\Dc\rho^* T\\
&\Ac_\Dc\Ac_\Dc\rho^*T\ar[ur]_{\,\,mult\circ Id_{\rho^*T}}&
}$$ so that, roughly speaking, $HH(\rho)^*$ is induction from $\Ac_\Cc$ to $\Ac_\Dc$ and  $HH(\rho)_*$ is the restriction from $\Ac_\Dc$ to $\Ac_\Cc$.

Of course the adjoint pair of functors also exists from formal considerations as a monoidal $\rho^*$ induces a map of cyclic objects in categories and its right adjoint $\rho_*$ induces a map of the corresponding cocyclic ones.

\begin{remark}\label{re:comm}
Observe that $HH(\rho)^* Tr\simeq Tr\rho^*$ or, equivalently, $U HH(\rho)_*\simeq \rho_* U$.
\end{remark}

\subsection{The  category of $B$-bimodules}\label{sss:bimodcat}
Let $B\in Alg(\Dc)$, we need to relate the Hochschild homologies of $\Dc$ and $\bimod_\Dc B$ by an adjoint pair of functors.  It will be useful to review a concept very much related to $\Delta_*$, namely that of internal Homs (see Section \ref{contratraces}).  Briefly, a monoidal category is said to be biclosed if the product has right adjoints (internal Homs), i.e., we have adjunctions: $$Hom_\Dc(X\ot T,M)\simeq Hom_\Dc(X,T\la M)\quad\text{and}\quad Hom_\Dc(T\ot X,M)\simeq Hom_\Dc(X, M\ra T).$$  Note that if $\Delta_*$ exists then we may demonstrate the biclosed property with a bit more effort (see Remark \ref{rem:biclosedfromadj}). 

The main ingredient here is the map $HH_\cdot(B,-):\bimod_\Dc B\to HH(\Dc)$ that sends a bimodule $S$ to its ``Hochschild homology" complex $Tr(S\ot B^{\ot-\bullet})$.  Note that applying $Tr$ allows us to bypass the problem that arises if $\Dc$ is not symmetric and so there is no Hochschild differential on $S\ot B^{\ot-\bullet}$ itself.  Since canonically $HH_\cdot(B,S\ot_B T)\simeq HH_\cdot(B,T\ot_B S)$ so $HH_\cdot(B,-)$ induces a map from $HH(\bimod_\Dc B)$, i.e., we have: $$\xymatrix{& \bimod_\Dc B\ar[dr]^{HH_\cdot(B,-)}\ar[dl]_{Tr} &\\
HH(\bimod_\Dc B)\ar@{-->}[rr]& & HH(\Dc)
}$$ where the induced dashed map has a right adjoint $-\ra B$. More precisely, for $M\in HH(\Dc)$ the object $M\ra B$ is not only a $B$-bimodule in $\Dc$, but is actually an object in $HH(\bimod_\Dc B)$ \cite{s}.  So we have the required pair of adjoint functors between the Hochschild homology categories of $\bimod_\Dc B$ and $\Dc$.

\begin{remark}
Returning to Remark \ref{re:extmix} we observe that, roughly speaking, we have an equivalence of mixed complexes $HH(A_\Cc\mmm)_*N\simeq Hom_{\bimod_\Cc A}(A, N\ra A)$.
\end{remark}

\section{Our special case: $\amod$ as an $\hmod$-bimodule}\label{specialcase}
We now apply the general observations of Section \ref{hochschildhomology} to the case of Hopf algebras. The goal here is to demonstrate that the well-known anti-Yetter-Drinfeld contramodules \cite{contra} arise as the Hochschild homology category of $\hmod$, the monoidal category of left $H$-modules.  Furthermore, the stable    anti-Yetter-Drinfeld contramodules are simply objects in the naive cyclic homology category of $\hmod$.

Let $A$ be an algebra in $\Mc^H$, then the structure morphism $\Delta_r:A\to A\ot H$ yields an adjoint pair of functors (this is reviewed below): $$\Delta_r^*:{_{A\ot H}\Mc}\leftrightarrows \amod:\Delta_{r*}.$$  If we set $\Cc=\hmod$ and $\Mc=\amod$ (the category of $A$-modules) then we have adjunctions between $\Mc\boxtimes\Cc$ and $\Mc$ giving the latter a structure of a right $\Cc$-module  such that the action has a right adjoint.  

\begin{remark}
The construction above is a special case of the considerations in Section \ref{sec:limit}.  Namely, if $A\in\Mc^H$ is an algebra then it yields a right module $A_{\Mc^H}\mmm$ over $\Mc^H$.  Furthermore, $\vect$ is an admissible $(\hmod,\Mc^H)$-bimodule and $Fun(\vect,A_{\Mc^H}\mmm)_{\Mc^H}$ is the right $\hmod$-module $\amod$ above.   Note that it need not, in general, be admissible, so that defining its Chern character would be problematic.
\end{remark}

Similarly we may consider an algebra $A$ now in ${}^H\!\!\Mc^H$ (in the sense that it has both  left and  right commuting coactions and both are compatible with the algebra structure), that yields a bimodule category $\Mc$ over $\Cc$ with both actions possessing right adjoints.

\begin{remark}\label{rem:tensorofcats}
Note that if $A$ is an algebra in $\Mc^H$ (or similarly in ${}^H\!\!\Mc$) then it will also yield an $\hmod$-bimodule category $\amod$ by adding the trivial (missing) $H$-coaction to $A$.  Furthermore, if $A\in\Mc^H$ and $B\in{}^H\Mc$ are algebras, then $\amod$ is a right $\hmod$-module and $\Bmod$ is a left $\hmod$-module, while $\amod\bt\Bmod={}_{A\ot B}\Mc$ is an $\hmod$-bimodule.
\end{remark}

\begin{remark}
If we let $A=H$ then all of the considerations of $\Mc=\amod$ apply also to $\Cc=\hmod$ itself.
\end{remark}

More precisely, observe that the coaction maps (using Sweedler notation) $\Delta_r:a\mapsto a_0\ot a_1$ and $\Delta_l:a\mapsto a_{-1}\ot a_0$, together with the counit map $\epsilon:H\to k$ and the coproduct $\Delta: h\mapsto h^1\ot h^2$, endow $A\ot H^{\ot \bullet}$ with the structure of a \emph{co-simplicial object in algebras}.  More precisely, we have face maps: 
\begin{align*}\delta_0(a\ot h_1\ot\cdots \ot h_n)&=a_0\ot a_1\ot h_1\ot\cdots\ot h_n\\
\delta_1(a\ot h_1\ot\cdots \ot h_n)&=a\ot h_1^1\ot h^2_1\ot\cdots\ot h_n\\
&\cdots\\
\delta_n(a\ot h_1\ot\cdots \ot h_n)&=a\ot h_1\ot\cdots\ot h_n^1\ot h^2_n\\
\delta_{n+1}(a\ot h_1\ot\cdots \ot h_n)&=a_0\ot h_1\ot\cdots\ot h_n\ot a_{-1}
\end{align*} and for $i=0,\cdots,n-1$ degeneracies: $$s_i(a\ot h_1\ot\cdots \ot h_n)=\epsilon(h_{i+1})a\ot h_1\ot\cdots\ot\widehat{h_{i+1}}\ot\cdots \ot h_n.$$   

First, a general observation: given a map of algebras $f:A\to B$ we have an adjoint pair of functors \begin{equation}\label{adjpair}(f^*,f_*)\end{equation}between their categories of modules as follows.  If $M$ is a $B$-module then for $m\in f^*(M)=M$ we let $a\cdot m=f(a)m$.  On the other hand if $N$ is an $A$-module we get that $f_*(N)=Hom_A(B,N)$ where we view $B$ as a left $A$-module via $f$, thus $b\cdot\varphi=\varphi(-b)$. 

Applying this to the cosimplicial object $A\ot H^{\ot \bullet}$ in algebras:  using $(-)^*$ we obtain that $\Mc\boxtimes\Cc^{\boxtimes\bullet}$ is a simplicial object in categories with functors possessing right adjoints.  Again, we may use $(-)_*$ to view $\Mc\boxtimes\Cc^{\boxtimes\bullet}$ as a co-simplicial object in categories with functors possessing left adjoints, and apply the discussion in Section \ref{hochschildhomology}.

The following is an easy, but key, lemma that will connect the general discussion to the case of Hopf algebras.  An important fact to note is the essential use of the antipode $S$ and its inverse $S^{-1}$ in the proof of  Lemma \ref{key lemma}.

\begin{lemma}\label{key lemma}
Let $V$ be an $H$-module. Consider $A\ot V$ as a left $A$-module via $\Delta_r$ and $V\ot A$ as a left $A$-module via $\Delta_l$, then we have isomorphisms of $A$-modules: \begin{equation}\label{iso1}A\ot V\simeq A\ot \underline{V}\quad\text{and}\quad V\ot A\simeq \underline{V}\ot A,\end{equation} where $\underline{V}$ denotes $V$ as purely a vector space (multiplicity).
\end{lemma}

\begin{proof}
Recall that $\Delta_r(a)=a_0\ot a_1$ while $\Delta_l(a)=a_{-1}\ot a_0$.  Consider the $\Delta_r$ case first.  Let $\varphi:A\ot\underline{V}\to A\ot V$ be such that $\varphi(a\ot v)=a_0\ot a_1 v$, then $\varphi(xa\ot v)=(xa)_0\ot(xa)_1 v=x_0a_0\ot x_1a_1v=x\cdot(a_0\ot a_1 v)=x\cdot\varphi(a\ot v)$, so $\varphi$ is indeed a map of $A$-modules.  Note that $\theta: A\ot V\to A\ot\underline{V}$ given by $\theta(a\ot v)=a_0\ot S(a_1)v$ is its inverse.

The $\Delta_l$ case is similar: $\varphi:\underline{V}\ot A\to V\ot A$ is given by $\varphi(v\ot a)=a_{-1}v\ot a_0$, while its inverse is $\theta(v\ot a)=S^{-1}(a_{-1})v\ot a_0$.
\end{proof}


\begin{corollary}\label{cortoold}
If $H$ is a Hopf algebra with an invertible antipode, $A$ as above, and $N$ is an $A$-module, then we have an isomorphism of $A\ot H$-modules: \begin{equation}\Delta_{r*}N\simeq Hom_k(H,N), \quad \text{with}\quad a\ot h\cdot\varphi=a_0\varphi(S(a_1)-h), \end{equation} and an isomorphism of $H\ot A$-modules: \begin{equation}\Delta_{l*}N\simeq Hom_k(H,N), \quad \text{with}\quad h\ot a\cdot\varphi=a_0\varphi(S^{-1}(a_{-1})-h). \end{equation}
\end{corollary}

\begin{proof}
Let $V=H$ and trace the isomorphisms of Lemma \ref{key lemma} between $Hom_A(A\ot H, N)$ and $Hom_k(H,N)$.
\end{proof}

\begin{remark} Using Corollary \ref{cortoold} in the case of $A=H$ we recover, via \eqref{adjointaction}, exactly the usual formulas for the biclosed structure on $\hmod$.

\end{remark}

Note that the cosimplicial object $\Mc\boxtimes\Cc^{\boxtimes\bullet}$, even without the \propA assumption, yields an $\Ac\in End(\amod)$ and a $u:Id\to\Ac$ which are described explicitly in the following:

\begin{corollary}\label{firstpart}
Let $H$ and $A$ be as above, the endofunctor $\Ac$ of $\amod$ obtained from the cosimplicial object $\Mc\boxtimes\Cc^{\boxtimes\bullet}$ is as follows:  if $N\in\amod$ then $\Ac(N)=Hom_k(H,N)$ with the $A$-action given by \begin{equation}\label{cor:action} a\cdot\varphi=a_0\varphi(S(a_1)-a_{-1}).\end{equation} Furthermore, the natural transformation $u:Id\to\Ac$ is given by $N\to Hom_k(H,N)$ where \begin{equation}\label{unit:eq}n\mapsto\varphi(h)=\epsilon(h)n.\end{equation} 
\end{corollary}

\begin{proof}
The claim \eqref{cor:action} follows immediately from Corollary \ref{cortoold}. For \eqref{unit:eq}, we note that  in our case, $N=\delta_1^*s_0^*s_{0*}\delta_{0*}N\to \delta_1^*\delta_{0*}N$ translates to $$N=Hom_{A\ot H}(A, Hom_{A}(A\ot H,N))=Hom_A(A\ot H,N)^H\to Hom_A(A\ot H,N)$$ with $n\mapsto \varphi(a\ot h)=\epsilon(h)an$ which identifies with \eqref{unit:eq} via Corollary \ref{cortoold}.
\end{proof}

\subsection{The \propA for algebras}\label{basechangeforalgebras}

Consider a commutative diagram of algebras: $$\xymatrix{& D &\\
B\ar[ur]^f & & C\ar[ul]_g\\
& A\ar[ur]_s\ar[ul]^t &
}$$ so that we have adjoint pairs of functors between the relevant categories of modules as in \eqref{adjpair}.  Furthermore, since $s^*g^*=t^*f^*$ we get a natural map $f^*g_*\to t_*s^*$ by adjunction, i.e., $f^*g_*\to t_*t^*f^*g_* = t_*s^*g^*g_*\to t_*s^*$.

\begin{lemma}\label{basiclemma}
The map $f^*g_*\to t_*s^*$ is an isomorphism if and only if the map \begin{equation}\label{themap}g\cdot f: C\ot_A B\to D\end{equation} with $c\ot b\mapsto g(c)f(b)$ is an isomorphism.
\end{lemma}

\begin{proof}
For $N\in\cmod$ observe that $t_*s^*(N)=Hom_A(B,N)=Hom_C(C\ot_A B,N)$ while $f^*g_*(N)=Hom_C(D,N)$ and the natural map is induced by the \eqref{themap}.
\end{proof}

\begin{lemma}\label{hasbasechange}
Let $H$ be a Hopf algebra (with an invertible antipode) and $A$ an $H$-bi-comodule algebra (as above).  Consider the monoidal category $\Cc=\hmod$ and its bimodule category $\Mc=\amod$.  Then, the cosimplicial object in categories with functors possessing left adjoints, $\Mc\boxtimes\Cc^{\boxtimes\bullet}$ has \propA\!\!. 
\end{lemma}
\begin{proof}
To prove \eqref{basechange1} it suffices by Lemma \ref{basiclemma} to demonstrate that the map $$A\ot H\ot_A A\ot H\to A\ot H\ot H$$ $$a\ot x\ot b\ot y\mapsto a_0b_0\ot a_1y\ot xb_{-1}$$ is an isomorphism. Note that the action of $A$ on the right copy of $A\ot H$ above is via $\Delta_r=\delta_0$.  Observe that by Lemma \ref{key lemma} the map $$A\ot H\ot \underline{H}\to A\ot H\ot_A A\ot \underline{H}\to A\ot H\ot_A A\ot H$$ $$a\ot x\ot y\mapsto a\ot x\ot 1\ot y\mapsto a\ot x\ot 1\ot y$$ is an isomorphism.  The composition of the latter followed by the former is easily seen to be $a\ot x\ot y\mapsto a_0\ot a_1y\ot x$  which is invertible with inverse: $a\ot x\ot y\mapsto a_0\ot y\ot S(a_1)x$.

The proof of \eqref{basechange2} is identical, i.e., apply the same argument to $$A\ot H\ot_A A\ot H\to A\ot H\ot H$$ $$a\ot x\ot b\ot y\mapsto ab_0\ot x^1y\ot x^2b_{-1}$$ to obtain $a\ot x\ot y\mapsto a\ot x^1y\ot x^2$ which is invertible using $S^{-1}$.

Finally \eqref{basechange3} concerns the map $$\delta_0\cdot\delta_3:A\ot H^2\ot_{A\ot H} A\ot H^2\to A\ot H^3$$ being an isomorphism.  The action of $A\ot H$ on the right copy of $A\ot H^2$ is via $\delta_0=\Delta_r\ot Id_H$.  The composition $A\ot H^2\ot\underline{H}\to A\ot H^3$ is now mapping $a\ot x\ot y \ot z$ to $a_0\ot a_1z\ot x\ot y$ which is invertible using $S$.
\end{proof}

With Lemma \ref{hasbasechange} in hand we can complete the Corollary \ref{firstpart} to an explicit description of $\Ac$ as a monad in our special case of $A$ and $H$.

\begin{proposition}\label{parttwo}
With $H$ and $A$ as in Corollary \ref{firstpart} and under the identification of $\Ac$ with $Hom_k(H,-)$, the monadic product on $\Ac$ is given by the following map: $$Hom_k(H,Hom_k(H,-))\to Hom_k(H,-)$$ $$\varphi(x)(y)\mapsto \varphi(h^1)(h^2).$$ 
\end{proposition}

\begin{proof}
Recall that the product on $\Ac$ is given by the composition: $$\delta_1^*\delta_{0*}\delta_1^*\delta_{0*}\leftarrow \delta_1^* \delta^*_2 \delta_{0*}\delta_{0*}=\delta_1^*\delta_1^*\delta_{1*}\delta_{0*}\to\delta_1^*\delta_{0*}$$ where the first map is an isomorphism by \eqref{basechange1}.  Let $N\in\amod$. In our case, the first identification is a composition: \begin{align*}Hom_A(A\ot H, Hom_A(A\ot H, N))&\simeq Hom_A(A\ot H\ot_A A\ot H, N)\\&\simeq Hom_A(A\ot H\ot H, N),\end{align*} the second is $$Hom_A(A\ot H\ot H, N)\simeq Hom_{A\ot H}(A\ot H\ot H, Hom_A(A\ot H,N)),$$ and the third is the evaluation at $1\in A\ot H\ot H$: $$Hom_{A\ot H}(A\ot H\ot H, Hom_A(A\ot H,N))\to Hom_A(A\ot H,N).$$  After careful tracing through the maps and identifications via Lemma \ref{key lemma} we obtain the result. 
\end{proof}

\begin{proposition}\label{prop:ayd}
Let $H$ be a Hopf algebra with an invertible antipode.  Then \begin{itemize}
\item $HH(\hmod,\hmod)$ is the category $\aydc$ of anti-Yetter-Drinfeld contramodules.
\item $HH(\hmod,\vect)$ is the category of $H$-contramodules. \end{itemize}
\end{proposition}
\begin{proof}
We apply the preceding discussion to an appropriate choice of $A$.  For the first case, let $A=H$ with the usual coactions via the coproduct $\Delta:H\to H\ot H$.  For the second, let $A=k$, with the trivial coactions, i.e., $\Delta_l=\Delta_r=u:k\to H$, where $u$ is the unit inclusion.  
\end{proof}

\begin{remark}
It is also possible to consider $A=H$ but instead of taking the usual $H$-coactions, to let the left one be trivial.  The consideration of the resulting $\hmod$-bimodule $\underline{\hmod}$, namely $HH(\hmod,\underline{\hmod})$ yields the contramodule variant of Hopf modules.   Since the latter is just $\hmod\bt_\hmod\vect$, so Remark \ref{rem:t} demonstrates that the Fundamental Theorem of Hopf contramodules also holds.
\end{remark}

Thus for a general $A$ as considered in this section, $HH(\hmod,\amod)$  yields the category (of generalized anti-Yetter-Drinfeld contramodules) $\widehat{\amod^H}$ that consists of $A$-modules and $H$-contramodules such that the two structures are compatible in the following sense: if $M$ is an $A$-module and $\alpha:Hom_k(H,M)\to M$ is the $H$-contramodule structure, then for $\varphi\in Hom_k(H,M)$ and $a\in A$ we have $$\alpha(a_0\varphi(S(a_1)-a_{-1}))=a\alpha(\varphi)$$ which is obtained directly from \eqref{cor:action}. 

\subsection{Cyclic homology}\label{sec:cyclicforalgebras}

Having addressed the various Hochschild homology categories, it is time to consider the cyclic version. Let $A=H$ so that we are in the setting of  Section \ref{sec:cyclicgen}.  It remains only to describe the element $\varsigma\in\Ac$, i.e., the map $N=\delta^*_1s^*_1s_{1*}\delta_{0*}N\to Hom_k(H,N)$ for $N\in\hmod$.  Recall that $\delta_{0*}N=Hom_k(H,N)$ has two $H$-actions: $H_1$ by $x^1\varphi(S(x^2)-)$ and $H_2$ by $\varphi(-x)$.  We observe that $\varsigma: N = Hom_k(H,N)^{H_1}\to Hom_k(H,N)$ via $\varsigma:n\mapsto (-)n\mapsto (-)n$.  So we see that (compare with \eqref{unit:eq}) \begin{align*}\varsigma:N&\to Hom_k(H,N)\\ n&\mapsto \varphi(h)=hn\end{align*} which together with Proposition \ref{prop:ayd} and Corollary \ref{cor:cycliccoh1} yields:

\begin{theorem}\label{thm:stableayd}
Let $H$ be a Hopf algebra with an invertible antipode, then $HC(\hmod)$ is the category of stable anti-Yetter-Drinfeld contramodules.
\end{theorem}

\section{Mixed anti-Yetter-Drinfeld contramodules and $ch(A)$}\label{subsec:mixedaYDcontra}
Here we pursue the heuristic outlined at the end of  Section \ref{sec:cyclicgen}, applied to the case of Section \ref{specialcase}.  One of the consequences of this point of view is an analogue of \cite{connes} for Hopf cyclic cohomology; this has been previously attempted in \cite{bivar} via a very different approach.  The idea is to replace  stable anti-Yetter-Drinfeld contramodules ($HC(\hmod)$ by Theorem \ref{thm:stableayd}) by  \emph{mixed} anti-Yetter-Drinfeld contramodules introduced below.  They are a generalization of mixed complexes, utilizing the $\varsigma$ in the Drinfeld center of $HH(\hmod)$.

\begin{definition}\label{S1act}
Let $\Mc$ be a $k$-linear category with $\varsigma\in Aut(Id_\Mc)$ equipping $M\in\Mc$ with $\varsigma_M\in Aut_\Mc(M)$.  Then denote by $\Mc^{S^1}$ the category of cohomological complexes of objects in $\Mc$ paired with a homotopy annihilating $1-\varsigma$, i.e., it consists of $(M^\bullet, d, h)$ with $d$ a cohomological differential, $h$ a homological differential and $$dh+hd=Id_{M^i}-\varsigma_{M^i}$$ for all $i$.
\end{definition}

Our main example of an $\Mc$ as above is $HH(\Cc)=HH(\Cc,\Cc)$.  Thus $HH(\Cc)^{S^1}$ is a modification of the naive $HC(\Cc)$.

\begin{definition}
We will call the objects of the category $HH(\hmod)^{S^1}$ \emph{mixed} anti-Yetter-Drinfeld contramodules.  Thus the classical stable anti-Yetter-Drinfeld contramodules correspond to objects concentrated in degree $0$.

\end{definition}

Note that this is not a needless generalization for two reasons.  First, as we will see below, this category is a natural recipient of the Chern character of an algebra, and second, the homotopy encodes the analogue of the flat structure on a $D_X$-module, for $X$ living in the algebraic geometry world.

For $A$ a unital associative algebra in $\hmod$  consider $Tr(A^{\ot\bullet+1})$ which is a paracyclic object in $\aydc$, i.e., a functor from $\Lambda^{op}_\infty$ to $\aydc$ in light of Remark \ref{rem:trace}.  More precisely, the faces are obtained via multiplication, the degeneracies via the unit, and finally the cyclic generator $\tau$ uses the cyclic property of $Tr$, namely Remark \ref{rem:trace}.    Since there is no guarantee that $\tau$ to the correct power is identity, it is not a cyclic object. 

Furthermore,  as in \cite{loday}, $Tr(A^{\ot-\bullet+1})$ has differentials $b$ and $B$ of degree $+1$ and $-1$ respectively such that $(bB+Bb)|_{Tr(A^{\ot n+1})}=1-\tau^{n+1}_n$ and the latter is $1-\varsigma_{Tr(A^{\ot n+1})}$  by Remark \ref{rem:cyctr}.  So we have that $(Tr(A^{\ot-\bullet+1}),b,B)$ is a mixed anti-Yetter-Drinfeld contramodule.

\begin{definition}\label{def:chernhmod}
Let $A\in Alg(\hmod)$, define the Chern character of $A$ in $\aydc^{S^1}$ as follows:
$$ch(A):=(Tr(A^{\ot-\bullet+1}),b,B).$$
\end{definition}

Recall that the above is not the first definition of $ch(A)$ that appears in this paper.  It is motivated by the initial Definition \ref{def:chern1} in light of Section \ref{ss:hhadj}.

\begin{remark}
We note that for $V,W\in\hmod$ we have $\tau_{V,W}:Tr(V\ot W)\to Tr(W\ot V)$  given by $$\tau_{V,W}\varphi(h)=1\ot h^2\circ\sigma_{V,W}\circ\varphi(h^1)$$ where $\sigma_{V,W}:V\ot W\to W\ot V$ sends $v\ot w$ to $w\ot v$.

Similarly, for $S\in\hhmod$ we have $\tau_S:Tr(\Delta^* S)\to Tr(\Delta^*\sigma S)$ given by $(\tau_S\varphi)(h)=(h^2\ot 1)\varphi(h^1)$.

\end{remark}

\subsection{A comparison theorem}\label{subsec:nvso}

Let $\Mc$ be as in the Definition \ref{S1act}. We have the following easy lemma that follows from a direct computation. The objects of $\vect^{S^1}$ below are the usual mixed complexes of \cite{kassel}.

\begin{lemma}
If $(M^\bullet, d_M, h_M)$ and $(N^\bullet, d_N, h_N)$ are in $\Mc^{S^1}$ then $$Hom_\Mc(M^\bullet, N^\bullet)^\bullet\in\vect^{S^1}$$ where $\vect$ is equipped with the trivial $S^1$ action.  More precisely, $Hom_\Mc(M^\bullet, N^\bullet)^i=\prod_j Hom(M^j,N^{i+j})$ and $d=[d,-]$, $h=[h,-]$, i.e., $$d\varphi=d_N\circ\varphi-(-1)^{deg(\varphi)}\varphi\circ d_M,\quad and \quad h\varphi=h_N\circ\varphi-(-1)^{deg(\varphi)}\varphi\circ h_M.$$
\end{lemma}

The lemma above actually defines module category internal Homs \cite{ostrik}.  More precisely, we have that $\Mc$ is a $\vect$-module category and so $K(\Mc)$, the category of cohomological complexes, is a $K(\vect)$-module, and similarly $\vect^{S^1}$ acts on $\Mc^{S^1}$ with $$(V^\bullet,d_V,h_V)\cdot (M^\bullet, d_M,h_M)=(V^\bullet\ot M^\bullet, d_V+d_M,h_V+h_M)$$  as expected.  Furthermore, if we set $$\underline{Hom}(M^\bullet,N^\bullet):=Hom_\Mc(M^\bullet, N^\bullet)^\bullet\in\vect^{S^1}$$ then \begin{equation}\label{inthom}Hom_{\vect^{S^1}}(V^\bullet, \underline{Hom}(M^\bullet,N^\bullet))\simeq Hom_{\Mc^{S^1}}(V^\bullet\cdot M^\bullet, N^\bullet).\end{equation}

\begin{lemma}\label{basic}
Let $(X^\bullet, d, h)\in\vect^{S^1}$ and $k$ be the trivial mixed complex then $$RHom_{\vect^{S^1}}(k, X^\bullet)\simeq X^\bullet[[y]]$$ with $deg(y)=+2$ and differential $d-hy$ on the latter.  The grading on $X^\bullet[[y]]$ is as follows: $(X^\bullet[[y]])^i=\prod_{j=0}^\infty X^{i-2j}y^j$.
\end{lemma}

\begin{proof}
As in \cite{kassel} replace $k$ by  $(k[x,\epsilon],\epsilon/x,\epsilon)$ with $deg(x)=-2$ and $deg(\epsilon)=-1$.
\end{proof}

\begin{remark}
More generally, for $X^\bullet, Y^\bullet\in \vect^{S^1}$ we have $$RHom_{\vect^{S^1}}(X^\bullet, Y^\bullet)\simeq Hom_k(X^\bullet, Y^\bullet)[[y]]$$ with $deg(y)=+2$ and differential $D=[d,-]-[h,-]y$, by \eqref{inthom} and Lemma \ref{basic}.
\end{remark}

\begin{proposition}\label{mp}
Let $(M^\bullet, d, h)\in\Mc^{S^1}$ be such that $M^{i>0}=0$ and $M^i$ is projective in $\Mc$; if $N^\bullet\in \Mc^{S^1}$ as well then $$RHom_{\Mc^{S^1}}(M^\bullet,N^\bullet)\simeq RHom_{\vect^{S^1}}(k, \underline{Hom}(M^\bullet,N^\bullet)).$$
\end{proposition}
\begin{proof}
Let $\Mc^{-S^1}$ denote the subcategory of $\Mc^{S^1}$ with complexes concentrated in non-positive degrees.  Similarly, $K^-(\Mc)$ denotes cohomological complexes in $\Mc$ in non-positive degrees.  Equip the latter with the projective model structure.  We recommend \cite{model} for an overview.

Consider $U:\Mc^{-S^1}\to K^-(\Mc)$ the forgetful functor (forgetting the homotopy $h$) and let $F:K^-(\Mc)\to \Mc^{-S^1}$ be $Cone(1-\varsigma)$, i.e., for $(A^\bullet, d)\in K^-(\Mc)$ we have $F(A^\bullet, d)=(A^\bullet[\epsilon],d+(1-\varsigma)\iota_\epsilon,\epsilon)$ where $deg(\epsilon)=-1$ and $\iota_\epsilon$ denotes contraction with $\epsilon$.  We note that $(F,U)$ is an adjoint pair of functors and the model structure on $\Mc^{-S^1}$ is transferred from that of $K^-(\Mc)$.

Observe that $FU(M^\bullet)=(M^\bullet[\epsilon],d+(1-\varsigma)\iota_\epsilon,\epsilon)$ is isomorphic to $(M^\bullet[\epsilon],d,\epsilon+h)$ via \begin{equation}\label{isoeqn}m_1+\epsilon m_2\mapsto m_1+hm_2+\epsilon m_2.\end{equation}  Furthermore we have a commutative diagram:$$
\xymatrix{FU(M^\bullet)\ar[d]^{\eqref{isoeqn}}\ar[r]^{\quad eval}& M^\bullet\ar[d]^{Id}\\
M^\bullet[\epsilon]\ar[r] & M^\bullet}
$$ where the top arrow is the evaluation map: $m_1+\epsilon m_2\mapsto m_1+hm_2$ and the bottom arrow is $m_1+\epsilon m_2\mapsto m_1$.  So $(M^\bullet[\epsilon],d,\epsilon+h)$ is cofibrant and thus $(M^\bullet[x,\epsilon],d+\epsilon/x,\epsilon+h)$ where $deg(x)=-2$ is a cofibrant replacement of $(M^\bullet,d,h)$.

It then follows that \begin{align*}RHom_{\Mc^{S^1}}(M^\bullet,N^\bullet)&\simeq Hom_{\Mc^{S^1}}((M^\bullet[x,\epsilon],d+\epsilon/x,\epsilon+h),N^\bullet)\\
&=Hom_{\Mc^{S^1}}((k[x,\epsilon],\epsilon/x,\epsilon)\cdot M^\bullet, N^\bullet)\\
&\simeq Hom_{\vect^{S^1}}((k[x,\epsilon],\epsilon/x,\epsilon),\underline{Hom}(M^\bullet,N^\bullet))\\
&\simeq RHom_{\vect^{S^1}}(k,\underline{Hom}(M^\bullet,N^\bullet))
\end{align*}

\end{proof}

We obtain the following as a corollary of the above.

\begin{theorem}\label{th:oldnew}
If $A$ is an algebra in $\hmod$ that is projective as an object and $M$ is a stable anti-Yetter-Drinfeld contramodule, then we can reinterpret the \emph{old} Hopf-cyclic cohomology: $$HC^{old}(A,M)\simeq RHom_{\aydc^{S^1}}(ch(A),M).$$
\end{theorem}

\begin{proof}
Since $A$ is projective then so is $A^{\ot n}$ as the monoidal category $\hmod$ has exact internal Homs (see \cite{ks} for example), i.e., $Hom_H(A^{\ot n},-)\simeq Hom_H(A,Hom^l(A^{\ot n-1},-))$.   Note that $Tr(A^{\ot n})$ is also projective since $(Tr, U)$ is an adjoint pair by Proposition \ref{prop:adj}.  By Proposition \ref{mp} we have $RHom_{\aydc^{S^1}}(ch(A),M)\simeq RHom_{\vect^{S^1}}(k, Hom_{\aydc}(ch(A), M))$ and as mixed complexes $Hom_{\aydc}(ch(A), M)\simeq Hom_H(A^{\ot -\bullet+1}, M)$ (the identification is via Proposition \ref{prop:adj} again) where the latter obtains the mixed complex structure from the cocyclic object structure on $Hom_H(A^{\ot n+1}, M)$ defined as usual \cite{contra}.  Since $HC^{old}(A,M)$ is defined as the cyclic cohomology of this cocyclic object, which is isomorphic to the cyclic cohomology \cite{kassel} of the associated mixed complex, which in turn is exactly the cohomology of $Hom_H(A^{\ot \bullet+1}, M)[[y]]$ with $D=b-By$, we are done by Lemma \ref{basic}.
\end{proof}

And so we are prompted to make the following definition:

\begin{definition}\label{defcoh}
Let $A\in Alg(\hmod)$ and $M^\bullet\in\aydc^{S^1}$ then $$HC^\bullet(A,M^\bullet):=RHom_{\aydc^{S^1}}(ch(A),M^\bullet).$$ This generalizes the old definition to algebras that are not projective as objects and to general \emph{mixed}  anti-Yetter-Drinfeld contramodule coefficients.
\end{definition}

\begin{remark}\label{ch1comp}
Let $(M^\bullet,d,h)\in\aydc^{S^1}$ then note that $((M^\bullet)^H,d,h)\in\vect^{S^1}$.  From the proofs of Theorem \ref{th:oldnew}, Proposition \ref{mp}, and the observation that $ch(1)\simeq Tr(1)$ we immediately obtain that if $1$ is projective in $\hmod$ (and thus so is everything) then $$RHom_{\aydc^{S^1}}(ch(1), M^\bullet)\simeq RHom_{\vect^{S^1}}(k,(M^\bullet)^H)$$ so that it is computed by $(M^\bullet)^H)[[y]]$ with $deg(y)=+2$ and the differential $d-hy$.
\end{remark}

More generally, taking $RHom_{HH(\Cc)^{S^1}}(ch(A),M)$ as the definition of cyclic cohomology of an algebra $A$ in a monoidal category $\Cc$ with coefficients in $M\in HH(\Cc)^{S^1}$, i.e., of $HC^\bullet(A,M)$, we get:  

\begin{remark}\label{re:extmix}
The adjunction for cyclic cohomology discussed in \cite{s} can be obtained from \eqref{eq:comp} since for $A\in Alg(\Cc)$ and $M\in HH(\Dc)^{S^1}$ we have: \begin{align*}HC^\bullet(\Mc^*A,M)&=RHom_{HH(\Dc)^{S^1}}(ch(\Mc^*A),M)\\
&\simeq RHom_{HH(\Dc)^{S^1}}(HH(\Mc)^*ch(A),M)\\
&\simeq RHom_{HH(\Dc)^{S^1}}(ch(A), HH(\Mc)_* M)=HC^\bullet(A, HH(\Mc)_* M).
\end{align*}  In particular, for $N\in HH(\Dc)^{S^1}$ we have $HC^\bullet(A,N)\simeq RHom_{mixed}(k, HH(A_\Cc\mmm)_*N)$ so that the generalized cyclic cohomology reduces to computing $Ext^n_{mixed}(k, T)$ for some mixed complex $T$.
\end{remark}

\subsection{A generalization of the cyclic double complex}\label{tsygan}
In this section we take a more hands on approach to defining cyclic cohomology with coefficients in a \emph{mixed} anti-Yetter-Drinfeld contramodule by generalizing the original Tsygan's construction of a double complex associated to a precocyclic object.  We make no claims about the compatibility of this method with the derived approach taken above in Definition \ref{defcoh}, though no doubt with some mild assumptions they agree.

Let $A$ be an associative, not necessarily unital, algebra in $\hmod$ and take $(M^\bullet,d,h)\in \aydc^{S^1}$ that is bounded as a complex. Consider $Hom_H(A^{\ot\bullet+1},M^j)$ for a fixed $j$, and note that we don't have degeneracies given by $1\in A$ anymore and $\tau^{n+1}_n$ is no longer $Id$ as $M^j$ is not assumed to be stable, however $\tau^{n+1}_n$ is induced by $\varsigma_{M^j}\in Aut(M^j)$.  Thus it is a preparacocyclic vector space.

Define a trigraded object $Hom_H(A^{\ot\bullet+1},M^\bullet)[x,\epsilon]$ with $deg(x)=+2$, $deg(\epsilon)=+1$ adjoined graded commutative variables. To be backward compatible, the exponent of $A$ gives the first index, total degree of $x,\epsilon$ provides the second index, and exponent of $M$ the third.  Thus the total degree of $Hom_H(A^{\ot i+1},M^j)x^k\epsilon^l$ is $i+j+2k+l$.

With the notation as in \cite{loday} we have operators:\begin{align*}\delta_1&=Nx\iota_\epsilon+(1-t)\epsilon\quad\text{with}\quad\delta_1^2=(1-\varsigma)x\\
\delta_2&=b\iota_\epsilon-b'\epsilon\iota_\epsilon\\
\delta_3&=d\\
\delta_4&=xh\quad\text{and}\quad [\delta_3,\delta_4]=\delta^2_1.\\
\end{align*}

Each operator is of total degree $1$ and except for the relations indicated above, everything else graded-commutes; so in particular squares to $0$.  We can thus take the total complex with the differential $$D=\delta_1+\delta_2+\delta_3-\delta_4$$ yielding cohomology groups as $D^2=0$, i.e., $\delta_4$ fixes the $\delta^2_1\neq 0$ defect caused by the instability of $M^j$s.

\begin{remark}
It is immediate that if $(M^\bullet,d,h)=M$, a single stable anti-Yetter-Drinfeld contramodule, then the above reduces to the Tsygan's double complex associated to a precocyclic object $Hom_H(A^{\ot\bullet+1},M)$, i.e., it calculates $HC^{old}(A,M)$.
\end{remark}

\section{The monoidal category of $H$-comodules}\label{sec:comodules}
We mentioned in Remark \ref{r:aydmod} that one may obtain anti-Yetter-Drinfeld modules, $\aydm$, by considering $\Delta_!$ instead of $\Delta_*$ in the case of $\hmod$.  However there is a more natural setting in which the module version of coefficients appears, namely that of the case: $\Cc=\Mc^H$, the monoidal category of $H$-comodules.  More precisely, the general methods in Section \ref{hochschildhomology} can be applied here as they were to the $\hmod$ case.  One may consider $H^{\bullet+1}$ as a cyclic object in coalgebras, and thus $(\Mc^H)^{\bullet+1}$ is a cyclic object in categories (with right adjoints).    For example, the multiplication $m: H^2\to H$ in $H$ induces an adjoint pair of functors: $m_!:\Mc^{H^2}\leftrightarrows\Mc^H:m^!$, where $m_!$ is the monoidal multiplication in $\Mc^H$ and $m^!$ its right adjoint.  Again, by considering $(\Mc^H)^{\bullet+1}$ as a cocyclic object in categories (with left adjoints) we can examine $HH(\Mc^H)$ as an inverse limit, and so as a module category over a  monad on $\Mc^H$. To be consistent with Section \ref{hochschildhomology} we let $\Delta^*=m_!$ and $\Delta_*=m^!$.  Unravelling the general definitions, we get: 

\begin{definition} The adjoint pair of functors $(\Delta^*,\Delta_*)$ is given by $$\Delta^*:\Mc^{H^2}\leftrightarrows\Mc^H:\Delta_*$$
\begin{itemize}
\item for $S\in\Mc^{H^2}$, given by $S\to S\ot H^2$ with $s\mapsto s_0\ot s_1\ot s_2$, we have $\Delta^* S=S$ with $S\to S\ot H$, given by $s\mapsto s_0\ot s_1 s_2$,

\item for $T\in\Mc^H$, given by $T\to T\ot H$ with $t\to t_0\ot t_1$, we have $\Delta_* T=H\ot T$ with $H\ot T\to H\ot T\ot H^2$, given by $x\ot t\mapsto x^2\ot t_0\ot t_1 S(x^1)\ot x^3$.
\end{itemize}  

\end{definition}

Thus for $T\in\Mc^H$: $$\Ac(T)=\Delta^*\sigma\Delta_* T=H\ot T$$ with the coaction $H\ot T\to H\ot T\ot H$:\begin{equation}\label{e:aydcond}x\ot t\mapsto x^2\ot t_0\ot x^3t_1S(x^1).\end{equation}  Furthermore, the map $\Ac^2(T)\to\Ac(T)$ is $m\ot Id_T: H^2\ot T\to H\ot T$ where $m$ is the multiplication in $H$. The identity of $\Ac$ is $Id: t\mapsto 1\ot t$ while the central element is $\varsigma: t\mapsto t_1\ot t_0$.  If $T$ is an $\Ac$-module then the action map $H\ot T\to T$ defines an action of $H$ on $T$.

So $T$ is a module over $\Delta^*\sigma\Delta_*$ if and only if $T\in\aydm$, i.e., $T$ is both an $H$-comodule and an $H$-module with the two structures compatible as dictated by the above (definitions were given by explicit formulas in \cite{JS} and \cite{HKRS2} independently).  Furthermore, as expected we have the free module functor $$Tr:\Mc^H\to\aydm$$ sending $T\in\Mc^H$ to $H\ot T$ with the comodule structure as in \eqref{e:aydcond} and the obvious $H$-action on the left on the first factor.  Naturally the functor $U$ the other way, that forgets the $H$-action, is right adjoint to $Tr$.

For any $M\in\aydm$ the $S^1$ action is given by $$\varsigma_M:m\mapsto m_1m_0$$ so that $\varsigma_{Tr(T)}(x\ot t)=xt_1\ot t_0$.

\begin{remark}
We have an analogue of Proposition \ref{prop:ayd}: $$HH(\Mc^H)=\aydm, HC(\Mc^H)=s\aydm,\quad and\quad HH(\Mc^H,\vect)=\hmod.$$   Note that the latter is very different from the $\hmod$ case, due to the fact that while $\vect$ is admissible as a right $\Mc^H$-module, it is not (in general) as a right $\hmod$-module. 
\end{remark}

\begin{definition}
We will call the objects of $\aydm^{S^1}$ the mixed anti-Yetter-Drinfeld \emph{modules}.
\end{definition}

And so forth, everything can be repeated here almost verbatim to define $ch(A)$, a mixed anti-Yetter-Drinfeld module associated to a unital associative algebra $A$ in $\Mc^H$.  The discussion of Section \ref{subsec:nvso} can be repeated here as well.

\section{Hochschild homology and adjunctions  in the Hopf setting}\label{sec:appl}
Let us examine the Hochschild homology adjunctions of Section \ref{ss:hhadj} in some examples in the setting where the monoidal categories arise as modules or comodules over Hopf algebras.

In particular, let us apply the constructions of Section \ref{sss:monfun} to the case of $\rho:K\to H$ a map of Hopf algebras which results in the obvious monoidal functor $\rho^*:\hmod\to\kmod$ with a right adjoint being coinduction, i.e., $\rho_* M=Hom_K(H,M)$.

We see that if $M\in{_K\aydc}$ with $K$-contraaction given by $\alpha: Hom(K,M)\to M$ then $HH(\rho)_* M=\rho_* M=Hom_K(H,M)$ with the $H$-contraaction obtained as the composition of two maps: $$Hom(H,Hom_K(H,M))\to Hom_K(H, Hom(K,M))\quad\text{with}\quad f(x\ot y)\mapsto g_f(h\ot k)$$ where $g_f(h\ot k)=f(S(h^3)\rho(k^1)h^1\ot\rho(k^2)h^2)$ and $$\alpha\circ-:Hom_K(H, Hom(K,M))\to Hom_K(H,M).$$

If $N\in{_H\aydc}$ with $H$-contraaction again denoted by $\alpha$ then $HH(\rho)^*N$ is given by the coequalizer of the diagram of free objects in ${_K\aydc}$: $$\xymatrix{Hom(K, Hom(H,N))\ar[rr]^{\alpha\circ-}\ar[d]_{(-\circ\rho)\circ-} & & Hom(K,N)\\
Hom(K, Hom(K,N))\ar[rr]_{\simeq} & & Hom(K\ot K,N)\ar[u]_{-\circ\Delta_K}
}$$ where $\Delta_K(k)=k^1\ot k^2$.

Let us return to the $\epsilon^*:\vect\to\hmod$ example and see the adjunctions between mixed complexes and $\aydc^{S^1}$ explicitly.  For $(W^\bullet, d, h)$ a mixed complex: $$HH(\epsilon)^*W^\bullet=(Hom(H,W^\bullet),d\circ -,h\circ -)$$ with action $x\cdot f=f(S(x^2)-x^1)$ and contraaction $-\circ\Delta_H:Hom(H,Hom(H, W^\bullet))\to Hom(H,W^\bullet)$.  Note that the individual $Hom(H, W^i)$ are stable.  Conversely, if $(M^\bullet,d,h)\in \aydc^{S^1}$ then \begin{equation}\label{eq:invar}HH(\epsilon)_*M^\bullet=(M^\bullet)^H \end{equation} is a mixed complex. 

\begin{remark}
The same $\rho$ induces $\rho^*:\Mc^K\to\Mc^H$ with a right adjoint.  Then we obtain an adjunction between $_K\aydm$ and $_H\aydm$  that already appeared in \cite{JS}, i.e.,  $HH(\rho)^*=Ind_K^H$ and $HH(\rho)_*=Res^H_K$.
\end{remark}

Furthermore, an example of the constructions in Section \ref{sss:bimodcat} is obtained by considering the identification of $\hmod$ with $\bimod_{\Mc^H}H$ via the fundamental theorem of Hopf modules.  Then we get the adjoint pair $((-)',\widehat{(-)})$ of functors: \begin{equation}\label{eq:theadj11}(-)':\aydc\leftrightarrows\aydm:\widehat{(-)}\end{equation} from \cite{s2}.  More precisely, we have $\vect$,  an admissible $(\hmod,\Mc^H)$-bimodule. Let $A$ be an $H$-module algebra and $M$ a mixed anti-Yetter-Drinfeld module.  We obtain a generalization of a result in \cite{s2} relating cyclic homologies between $H$-module algebras and $H$-comodule algebras. Namely, by Remark \ref{re:extmix} we have: $$HC^{\bullet,H}(A\rtimes H, M)=HC^{\bullet,H}(\vect^* A, M)\simeq HC_H^\bullet(A,HH(\vect)_*M)= HC_H^\bullet(A,\widehat{M}).$$

\section{Appendix}\label{sec:app}
We collect here some material that is not absolutely necessary for the paper, yet without which it would, we feel, not be complete.

\subsection{The case of $QC(X)$}\label{ss:qcx}

The discussion contained in this section is non-rigorous and is provided for perspective only. Let us briefly illustrate some of the considerations of this paper with the case of $\Cc=QC(X)$, quasicoherent sheaves on $X$, done properly in \cite{it, loop1, loop2}.  For simplicity, assume that $X$ is an affine scheme.  The diagonals and projections endow $X^{\bullet+1}$ with the structure of a cocyclic object, as usual.    Let $\Delta:X\to X\times X$ be the diagonal which induces the adjoint pair $(\Delta^*,\Delta_*)$ between $QC(X^2)$ and $QC(X)$.  Similarly, $\Cc^{\bt\bullet+1}=QC(X^{\bullet+1})$ is a cocyclic object in categories with left adjoints.

For $T\in QC(X)$ we have $\Delta^*\sigma\Delta_* (T)\simeq\Omega^{-\bullet}_X\ot_{\Oc_X}T$.  The multiplication is given by that of the graded commutative algebra $\Omega^{-\bullet}_X$, i.e., functions on $LX$, the loop space of $X$, also known as $TX[-1]$, the odd tangent space to $X$.  To describe the unit $1$ and central element $\varsigma$ we need the following discussion.

Let $C^{-\bullet}(\Oc_X, T)$ denote the Hochschild homology complex (in negative degrees) of $\Oc_X$  with coefficients in $T$ viewed as a diagonal bimodule.  It is classical that $\Omega^{-\bullet}_X\ot_{\Oc_X}T$ and $C^{-\bullet}(\Oc_X, T)$ are equivalent as complexes in $QC(X)$, where the action of $\Oc_X$ on $C^{-\bullet}(\Oc_X, T)$ is via the left action on $T$.  Consider $C^{-\bullet}(\Oc_X, \underline{T}\boxtimes \Oc_X)$ and $C^{-\bullet}(\Oc_X, \underline{\Oc_X}\boxtimes T)$ with the obvious non-diagonal bimodule structures on the coefficients, but take care to note that the action of $\Oc_X$ on the complexes is via the underlined components of the coefficients.   Both $\underline{T}\boxtimes \Oc_X$ and $\underline{\Oc_X}\boxtimes T$ map to $T$ via the $\Oc_X$ action on $T$ and these maps induce morphisms of the corresponding complexes.  Thus $$T\stackrel{q-iso}{\longleftarrow}C^{-\bullet}(\Oc_X, \underline{T}\boxtimes \Oc_X)\to C^{-\bullet}(\Oc_X, T)$$ is the identity of $\Delta^*\sigma\Delta_*$ while $$T\stackrel{q-iso}{\longleftarrow}C^{-\bullet}(\Oc_X, \underline{\Oc_X}\boxtimes T)\to C^{-\bullet}(\Oc_X, T)$$ is the $\varsigma$ of $\Delta^*\sigma\Delta_*$.

Note that the identity map is easy to describe more explicitly as $\Oc_X=\Omega^0_X\to\Omega_X^{-\bullet}$, while $\varsigma$ is closely related to $J_1\Oc_X\in Hom_{\Oc_X}(\Oc_X,\Omega^1_X[1])$, the first order jets of $\Oc_X$.  As is explained in \cite{loop1, loop2} we have that $HH(QC(X))$ consists of modules over the graded commutative algebra $\Omega_X^{-\bullet}$, i.e., is equivalent to $QC(LX)$ while $HC(QC(X))$ is formed by the modules over the differential graded commutative algebra $(\Omega_X^{-\bullet},d_{DeRham})$ and is closely related to $D_X$-modules.  In particular, the trivialization of the $S^1$-action on a particular object pulled back from $QC(X)$ to the Hochschild homology category is related to the flat structure on it, i.e., upgrading it to a $D_X$-module.  More precisely, a $D_X$-module $M$, i.e., an $\Oc_X$-module with a compatible action of vector fields, is viewed as $\Omega_X^{-\bullet}\ot_{\Oc_X}M$ with the obvious $\Omega_X^{-\bullet}$-module structure and homotopy given by $d_M$.

Given a map $f:X\to Y$, we have the monoidal functor $f^*:\Oc_Y\mmm\to\Oc_X\mmm$ with the right adjoint $f_*$.  The adjunction of Hochschild homology categories then comes from the pullback of forms: $f^*:\Omega_Y\to\Omega_X$ via induction and restriction.  Also note that the pair $(f^*,f_*)$ gives us the correct, in the sense of Section \ref{ss:lower}, adjoints on algebras. More precisely, we get that the functors $f^*:\Oc_Y\text{-alg}\leftrightarrows\Oc_X\text{-alg}:f_*$ induce an equivalence $$Fun (A\ot_{\Oc_Y}\Oc_X\mmm, B\mmm)_{\Oc_X\mmm}\simeq Fun(A\mmm, B\mmm)_{\Oc_Y\mmm},$$ since as a right $\Oc_Y\mmm$ module category $B_{\Oc_X\mmm}\mmm=B_{\Oc_Y\mmm}\mmm=B\mmm$, as $B\ot_{\Oc_Y}M\simeq B\ot_{\Oc_X}\Oc_X\ot_{\Oc_Y}M$.

We observe that the case when we replace $QC(X)$ by modules over a Hopf algebra $H$, though the monoidal category $\hmod$ is no longer symmetric nor braided, is much less technologically sophisticated since $\Delta^*$ is exact there.  All of the structures involved are very explicitly describable with none of the subtlety of morphisms in a derived category obscuring the picture.

\subsection{Contratraces}\label{contratraces}

In this section we relate the present context to the one considered in \cite{ks} where biclosed monoidal categories (i.e., possessing both left and right internal Homs) were the starting point for the development of cyclic cohomology with coefficients.    Recall that internal Homs are right adjoints to the monoidal products, i.e., $$Hom_\Cc(X\ot Y, Z)=Hom_\Cc(X,\underline{Hom}^l(Y,Z))=Hom_\Cc(Y,\underline{Hom}^r(X,Z)).$$

A minor generalization of what is considered in \cite{ks} is the category $\Zc_\Cc(\Mc^{op})$.  More precisely, given a $\Cc$-bimodule category $\Mc$ we assume that there exist ``internal Homs" for the actions, i.e., $Hom_\Mc(M\cdot C,N)=Hom_\Mc(M,C\la N)$ and similarly for the left action.  Then  $\Zc_\Cc(\Mc^{op})$ is  the center, in the usual sense, of the $\Cc$-bimodule category $\Mc^{op}$ where the action is given by $\la$ and $\ra$.  Recall that it consists of objects $M$ together with natural isomorphisms for every $C\in\Cc$: \begin{equation}\label{tauadj}\tau_C:C\la M\to M\ra C,\end{equation}  subject to the unit condition: $$\tau_1=Id_M,$$ and the associativity condition (for $C,D\in\Cc$): $$(\tau_C\ra Id)\circ(Id\la \tau_D)=\tau_{C\ot D}.$$  The maps in $\Zc_\Cc(\Mc^{op})$ are just morphisms in $\Mc$ that are compatible with the $\tau_C$'s.

It is perhaps easier to understand the conditions above when equivalently reformulated in terms of the contravariant functor from $\Mc$ to vector spaces (that $M$ represents).  More precisely, consider the $\Cc$-bimodule category of contravariant functors from $\Mc$ to $\vect$, with the usual action: $$C\cdot G(-)=G(-\cdot C)\quad and\quad G(-)\cdot C=G(C\cdot -).$$ Note that due to the ``biclosed" property of the actions on $\Mc$, the representable functors form a sub-bimodule category, and it is $\Mc^{op}$. If we set $F(-)=Hom_\Mc(-,M)$ then we get isomorphisms $\tau_C:C\cdot F\to F\cdot C$ such that $\tau_1=Id$, but more importantly, to pass $C\ot D$ through $F$ via $\tau$ we may pass $D$ first, followed by $C$.

Observe that if the $\Cc$-bimodule $\Mc$ is such that both actions have right adjoints (in the sense of Section \ref{hochschildhomology}) then we can do the following.  Let $\sigma:\Cc\bt\Mc\to\Mc\bt\Cc$ be the flip. In our case we have $\Cc_\bullet=\Mc\boxtimes\Cc^{\boxtimes\bullet}$ fibered over the simplex category $\Delta$ so that the inverse limit consists of $M\in\Mc$ together with an isomorphism : \begin{equation}\label{thetau}\tau:\Delta_{r*}M\to\sigma\Delta_{l*}M\end{equation} in $\Mc\boxtimes\Cc$ subject to a unit and an associativity condition that proves the Lemma \ref{contracoh} below. 

For $M,N\in\Mc$ and $C\in \Cc$ we get $$Hom_\Mc(N\cdot C,M)\simeq Hom_{\Mc\boxtimes\Cc}(N\boxtimes C,\Delta_{r*}M)\simeq Hom_\Mc(N,C\lba\Delta_{r*}M)$$ and so we see that \begin{equation}\label{adjointaction}C\la M= C\lba\Delta_{r*}M\end{equation} demonstrates the existence of the left adjoint action of $\Cc$ on $\Mc^{op}$.  The right action on $\Mc^{op}$ is obtained from $\Delta_{l*}$ in an identical manner.  Thus for such $\Mc$ as we consider in this paper we have that the $\Cc$-actions on them are also ``biclosed".

The definition of $C\lba T \in \Mc$ for $C\in\Cc$ and $T\in\Mc\boxtimes\Cc$ needs a few words of explanation. In the setting of Section \ref{specialcase} that is of primary interest to us it is completely straightforward.  Namely, for $N\in\amod$, $C\in\hmod$, and $T\in{{}_{A\ot H}\Mc}$ we have $$Hom_{A\ot H}(N\ot C, T)=Hom_A(N, Hom_H(C,T))$$ where $C\lba T=Hom_H(C,T)$ has an obvious $A$-action obtained from that on $T$.  Note that $H\lba T=T$ as both $A$-modules and $H$-modules, where the action of $H$ on the LHS is obtained from the action of $H$ on itself on the right.  This allows for the recovery of $\Delta_{r*}M\to\sigma\Delta_{l*}M$ of \eqref{thetau} from the collection of $C\la M\to M\ra C$.

\begin{lemma}\label{contracoh}
Under the assumption that $\Cc$ and $\Mc$ are as in Section \ref{specialcase} we have that $$HH(\Cc,\Mc)=\Zc_\Cc(\Mc^{op}).$$
\end{lemma}

\begin{proof}
This is almost a tautology since for $M\in HH(\Cc,\Mc)$ the $\tau$ of \eqref{thetau} equips the representable $F(-)=Hom_\Mc(-,M)$ with the necessary $\tau_C$'s for $C\in\Cc$ via \begin{align*}C\cdot F&=Hom_\Mc(-\cdot C,M)\simeq Hom_{\Mc\boxtimes\Cc}(-\boxtimes C,\Delta_{r*}M)\\&\simeq Hom_{\Mc\boxtimes\Cc}(-\boxtimes C,\sigma\Delta_{l*}M)\simeq Hom_\Mc(C\cdot -,M)=F\cdot C,\end{align*} while the unit and associativity conditions ensure the same for $\tau_C$'s.

The above argument can almost be reversed (via Yoneda) except that not every object of $\Mc\boxtimes\Cc$ is of the form $N\boxtimes C$.  This is easily rectified under the assumptions of Section \ref{specialcase} since the $\tau_C$ of \eqref{tauadj} when used with $C=H$ and its naturality, actually equals (recovers/defines) that of \eqref{thetau}, i.e., $\tau=\tau_H$.

Let us also sketch the equivalence of the centrality associativity to that of \eqref{associativity}.  The unit condition is similar.  Observe that $\delta_{1*}(\tau)$ identifies $Hom_\Mc(N\cdot(C\ot D),M)$ with $Hom_\Mc((C\ot D)\cdot N,M)$ via \begin{align*}Hom&_\Mc(N\cdot(C\ot D),M)=Hom_\Mc(\delta_0^*\delta_1^*(N\boxtimes C\boxtimes D),M)\\&=Hom_{\Mc\boxtimes\Cc^2}(N\boxtimes C\boxtimes D,\delta_{1*}\delta_{0*}M)
\stackrel{\delta_{1*}(\tau)}{\to}Hom_{\Mc\boxtimes\Cc^2}(N\boxtimes C\boxtimes D,\delta_{1*}\delta_{1*}M)\\
&=Hom_\Mc((C\ot D)\cdot N,M).
\end{align*}Similarly $\delta_{0*}(\tau)$ identifies $Hom_\Mc(N\cdot(C\ot D),M)$ with $Hom_\Mc(D\cdot N\cdot C,M)$, and $\delta_{2*}(\tau)$ identifies the latter with $Hom_\Mc((C\ot D)\cdot N,M)$.
\end{proof}

Note that Lemma \ref{contracoh} holds even in the case when $H$ is a bialgebra since the presence of the antipode $S$ was never used.

\begin{remark}\label{rem:biclosedfromadj}
Even though a correct definition of $C\lba T$ can be given (in either the $H$-module, or $H$-comodule case), and so the biclosed property would follow from the existence of right adjoints to the actions, it is more difficult to define than the adjoints.  On the other hand the monadic approach bypasses the biclosed property necessary to speak of $\Zc_\Cc(\Cc^{op})$, while dealing with the action adjoints directly.  The monadic approach yields immediately the usual descriptions of anti-Yetter-Drinfeld modules and contramodules in the comodule and module case respectively.
\end{remark}

\bibliography{hopfext}{}
\bibliographystyle{plain}
\bigskip

\noindent Department of Mathematics and Statistics,
University of Windsor, 401 Sunset Avenue, Windsor, Ontario N9B 3P4, Canada

\noindent\emph{E-mail address}:
\textbf{ishapiro@uwindsor.ca}

\end{document}